\documentclass{article}

\usepackage{geometry}
\usepackage{graphicx}
\usepackage{amsmath}
\usepackage{amssymb}
\usepackage{color}
\usepackage{hyperref}
\usepackage{epsfig}
\usepackage{tikz}
\usepackage{pgfplots}
\usepackage{caption}
\usepackage{float}
\usepackage{subcaption}
\usepackage{enumitem}

\usepackage{lineno}

\usepackage{ulem}
\usepackage{cancel}

\usepackage{bbm}
\usepackage{amsfonts}
\usepackage{amsmath, amsthm, amsfonts, amssymb, color}
\usepackage{mathrsfs}
\usepackage{color}

\allowdisplaybreaks

\newtheorem{thm}{Theorem}[section]
\newtheorem{cor}[thm]{Corollary}
\newtheorem{lem}[thm]{Lemma}
\newtheorem{prp}[thm]{Proposition}
\newtheorem{rmk}[thm]{Remark}

\theoremstyle{definition}
\newtheorem{defn}{Definition}[section]
\newcommand{\scr}[1]{\mathscr #1}
\definecolor{wco}{rgb}{0.5,0.2,0.3}

\numberwithin{equation}{section} \theoremstyle{remark}

\newcommand{\ua}{\uparrow}

\newcommand{\RR}{{\mathbb R}}

\newcommand{\notiz}[1]{}

\title{ First order convergence of Milstein schemes for
McKean--Vlasov equations and interacting particle systems}   
\author{Jianhai Bao\footnote{Center for Applied Mathematics, Tianjin University, 300072 Tianjin, China, E-mail: jianhaibao13@gmail.com. } \ , Christoph Reisinger\footnote{Mathematical Institute, University of Oxford, Andrew Wiles Building, Woodstock Road, Oxford, OX2 6GG, UK, E-mail: christoph.reisinger@maths.ox.ac.uk.
} \ , Panpan Ren\footnote{Department of Mathematics, City University of Hong Kong, Kowloon, Hong Kong, China, E-mail: rppzoe@gmail.com.} \ , Wolfgang Stockinger\footnote{Mathematical Institute, University of Oxford, Andrew Wiles Building, Woodstock Road, Oxford, OX2 6GG, UK, E-mail: wolfgang.stockinger@maths.ox.ac.uk.}
}
\date{}
\begin{document}
\allowdisplaybreaks
\def\R{\mathbb R}  \def\ff{\frac} \def\ss{\sqrt} \def\B{\mathbf
B}
\def\N{\mathbb N} \def\kk{\kappa} \def\m{{\bf m}}
\def\ee{\varepsilon}\def\ddd{D^*}
\def\dd{\delta} \def\DD{\Delta} \def\vv{\varepsilon} \def\rr{\rho}
\def\<{\langle} \def\>{\rangle} \def\GG{\Gamma} \def\gg{\gamma}
  \def\nn{\nabla} \def\pp{\partial} \def\E{\mathbb E}
\def\d{\, \text{\rm{d}}} 

\def\bb{\beta} \def\aa{\alpha} \def\D{\scr D}
  \def\si{\sigma} \def\ess{\text{\rm{ess}}}
\def\beg{\begin} \def\beq{\begin{equation}}  \def\F{\scr F}
\def\Ric{\text{\rm{Ric}}} \def\Hess{\text{\rm{Hess}}}
\def\e{\text{\rm{e}}} \def\ua{\underline a} \def\OO{\Omega}  \def\oo{\omega}
 \def\tt{\tilde} \def\Ric{\text{\rm{Ric}}}
\def\cut{\text{\rm{cut}}} \def\P{\mathbb P} \def\ifn{I_n(f^{\bigotimes n})}
\def\C{\scr C}   \def\G{\scr G}   \def\aaa{\mathbf{r}}     \def\r{r}
\def\gap{\text{\rm{gap}}} \def\prr{\pi_{{\bf m},\varrho}}  \def\r{\mathbf r}
\def\Z{\mathbb Z} \def\vrr{\varrho} \def\ll{\lambda}
\def\L{\scr L}\def\Tt{\tt} \def\TT{\tt}\def\II{\mathbb I}
\def\i{{\rm in}}\def\Sect{{\rm Sect}}  \def\H{\mathbb H}
\def\M{\scr M}\def\Q{\mathbb Q} \def\texto{\text{o}} \def\LL{\Lambda}
\def\Rank{{\rm Rank}} \def\B{\scr B} \def\i{{\rm i}} \def\HR{\hat{\R}^d}
\def\to{\rightarrow}\def\l{\ell}\def\iint{\int}
\def\EE{\scr E}\def\no{\nonumber}
\def\A{\scr A}\def\V{\mathbb V}\def\osc{{\rm osc}}
\def\BB{\scr B}\def\Ent{{\rm Ent}}
\def\U{\scr U}\def\8{\infty} \def\si{\sigma}\def\1{\lesssim}

\maketitle 
\begin{abstract}
\noindent
In this paper, we derive fully implementable first order time-stepping schemes for McKean--Vlasov stochastic differential equations (McKean--Vlasov SDEs), allowing for a drift term with super-linear growth in the state component.
We propose Milstein schemes
for a time-discretised interacting particle system associated with the McKean--Vlasov equation and prove strong convergence of order 1 and moment stability, taming the drift if only a one-sided Lipschitz condition holds.
To derive our main results on strong convergence rates, we make use of calculus on the space of probability measures with finite second order moments. In addition, numerical examples are presented which support our theoretical findings.  
\end{abstract}

\section{Introduction} 
A McKean--Vlasov equation (introduced by H.\ McKean \cite{MK}) for a $d$-dimensional process $X$ is an SDE where the underlying coefficients depend on the current state $X_t$ and, additionally, on the  law of $X_t$, i.e., 
\begin{equation}\label{McKean--VlasovLimit}
    \mathrm{d}X_t =b(X_t, \mathscr{L}_{X_t}) \, \mathrm{d}t + \sigma(X_t, \mathscr{L}_{X_t}) \, \mathrm{d}W_t, \quad X_0 = \xi,
\end{equation}
where $W$ is an $m$-dimensional standard Brownian motion, $ \mathscr{L}_{X_t}$ denotes the marginal law of the process $X$ at time $t \geq 0$ and $\xi$ is an $\mathbb{R}^d$-valued random variable.
We omit an explicit dependence of the coefficients on $t$ for brevity, but our results easily generalise to this case.

The existence and uniqueness theory for strong solutions of McKean--Vlasov SDEs with coefficients of linear growth and Lipschitz type conditions (with respect to the state and the measure) is well-established (see, e.g., \cite{AS}). For further existence and uniqueness results on weak and strong solutions of McKean--Vlasov SDEs we refer the reader to \cite{BBP, HSS, MVA} and the references cited therein. Also, in the case of super-linear growth it is known that a McKean--Vlasov SDE admits a unique strong solution \cite{RST}, assuming a so-called one-sided Lipschitz condition for the drift term, see item (1) of assumption ({\bf A}$_b^1$) in Section \ref{subsec:prelims}.

McKean--Vlasov SDEs have numerous applications, for instance, in the social and natural sciences.
These include fundamental models 
in neuroscience, such as the Hodgkin-Huxley model (see \cite{BFFT}) for neuron activation, or in biology and chemistry, such as
the Patlak-Keller-Segel equations describing, e.g., chemotactic interactions (see \cite{KS}) and long-chain polymers (see \cite{P}).
The aforementioned examples all have drift terms which do not exhibit the classical global Lipschitz conditions on the coefficients of the SDE.


The simulation of McKean--Vlasov SDEs typically involves two steps: First, at each time $t$, the true measure $ \mathscr{L}_{X_t}$ is approximated by the empirical measure 
\begin{equation*}
 \mu_t^{X^{\cdot,N}}(\mathrm{d}x) := \frac{1}{N}\sum_{j=1}^{N} \delta_{X_t^{j,N}}(\mathrm{d}x),
\end{equation*}    
where $\delta_{x}$ denotes the Dirac measure at point $x$ and $(X^{i,N})_{i=1, \ldots, N}$ (so-called interacting particles) is the solution to the $\mathbb{R}^{dN}$-dimensional SDE
\begin{equation*}
\mathrm{d}X_t^{i,N} = b(X_t^{i,N}, \mu_t^{X^{\cdot,N} }) \, \mathrm{d}t + \sigma(X_t^{i,N},\mu_t^{X^{\cdot,N} }) \, \mathrm{d}W_t^{i}, \quad X_{0}^{i,N} = X_0^{i}.
\end{equation*}
Here, $W^{i}$ and $X_{0}^{i}$, $i=1, \ldots, N$ are independent Brownian motions (also independent of $W$) and i.~i.~d.\ random initial values with $\mathscr{L}_{X_0^{i}} = \mathscr{L}_{X_0}$, respectively.
In the second step, one needs to introduce a reasonable time-stepping method to discretise the particle system $(X^{i,N})_{i=1, \ldots, N}$ over some finite time horizon $[0,T]$. 

An Euler scheme for the particle system is introduced in \cite{BT}, and the strong convergence to the solution of the McKean--Vlasov equation, of order 1/2 in the time-step and also 1/2 in the number of particles,  is shown under global Lipschitz assumptions for the coefficients with respect to the state and measure (for $d=m=1$). The convergence in $N$ is often referred to as propagation of chaos.

The results in \cite{BT}
have recently been extended to the case where only a one-sided Lipschitz condition for the drift holds with respect to the state, while a global Lipschitz condition is still assumed for all other dependencies, by using so-called tamed and implicit schemes (in \cite{RES}) or adaptive schemes (in \cite{RS}).

An approximation of order 1/2 in the timestep is also given in \cite{BT} for the density and cumulative distribution function, and this is improved to order 1 in \cite{ANK}.
In this paper, in contrast, we are concerned with the strong convergence of the approximations to the process itself.

We complement the work in \cite{BT} and \cite{RES} by introducing stable first order time-stepping schemes (Milstein schemes) for a particle system associated with McKean--Vlasov SDEs, allowing for a drift coefficient which grows super-linearly in the state component. Here, we will prove moment stability of the time-discretised particle system and strong convergence of order $1$.

Our new Milstein scheme reveals that a term involving the Lions derivative (abbreviated by $L$-derivative) of the diffusion coefficient with respect to the (empirical) measure is necessary to obtain the strong convergence result. The $L$-derivative of functions on $\mathscr{P}_2(\mathbb{R}^d)$ was introduced by P.\ -L.\ Lions in his lectures \cite{PLI} at Coll\`{e}ge de France. This is of theoretical interest and demonstrates the inherent difference of McKean--Vlasov SDEs to classical SDEs with regard to higher order time-stepping schemes.

The main difficulty presented by the super-linearity is to prove the stability of the proposed scheme, and we adopt here the taming approach given in \cite{HJK} for standard SDEs.
In the present context, we require pathwise estimates for each particle, which are challenging to obtain as the particles interact through the appearance of the empirical distribution of the particle system in the coefficients. 

The convergence results for the time-stepping scheme hold for a fixed dimension $N$, and are robust in $N$. To obtain error bounds with respect to the solution of
\eqref{McKean--VlasovLimit}, these have to be supplemented by propagation of chaos results from \cite{RES} to bound the error in $N$ from the particle approximation.
In practice, many McKean--Vlasov equations are motivated by large but finite particle systems, and our time-stepping schemes are directly applicable in that case.

We will demonstrate that the terms involving the measure derivative are only significant for fixed finite $N$, but vanish with the same order of $N$ as the terms from the particle approximation. Omitting these terms to obtain a simplified scheme for large $N$  is practically useful as they require the simulation of L{\'e}vy areas and are the computationally most expensive part of the scheme.

On a side note, in the special case $N=1$ our work gives the first order convergence of tamed Milstein schemes for standard SDEs with super-linear drift.
It is observed in \cite{HJK2} that
already in this setting of classical SDEs (i.e., where the coefficients have no measure-dependence), the explicit Euler--Maruyama scheme (see, e.g., \cite{KP}) is not appropriate 
in the presence of drift terms with super-linear growth.
To overcome this problem, several stable time-discretisation methods, including a tamed explicit Euler and Milstein scheme \cite{HJK, SA, GW}, an explicit adaptive Euler--Maruyama method \cite{FG}, a truncated Euler method \cite{XM2} and an implicit Euler scheme \cite{HMS}, have been introduced.
In \cite{GW}, a tamed Milstein scheme is introduced and convergence is proven under a commutativity assumption for the diffusion matrix, such that the L\'{e}vy area vanishes. Our technique of handling the L\'{e}vy area terms in a pathwise sense allows the analysis of the general case without this assumption. While \cite{KUMAR} introduces a tamed Milstein scheme for SDEs with super-linearly growing drift and diffusion without assuming a commutativity assumption, the taming there is required to be stronger than the taming approach proposed in \cite{HJK}, which makes the stability analysis of the schemes easier, but can result in inferior numerical performance (see Remark \ref{rem:taming} and Section \ref{Section:Sec4} for details).  

In summary, the main contributions of this paper are the following:
\begin{itemize}
\item
derivation of a Milstein scheme for particle systems associated with McKean--Vlasov equations using Lions calculus on measure space;
\item
proof of moment stability and first order uniform (in time) convergence in the time-step using a tamed scheme if only a one-sided Lipschitz condition holds for the drift;
\item
estimates of the terms involving measure derivatives, showing that these are essential for small particle systems but negligible to approximate the mean-field limit;
\item
detailed numerical tests supporting the theoretical findings.
\item
The main results extend those on tamed Milstein schemes (for certain taming approaches) for standard SDEs by eliminating commutativity conditions.
\end{itemize}

A tamed Milstein scheme, with stronger taming than used in our work, was developed under a similar set of assumptions independently and in parallel in \cite{K}. The authors only state a pointwise (in time) strong $L_2$ convergence rate of order 1, albeit under slightly weaker differentiability conditions. The present work also goes beyond \cite{K} by analysing the asymptotic behaviour (in terms of the number of particles) of the $L$-derivative terms in the schemes and by providing detailed numerical illustrations. 

The remainder of this article is organised as follows: In Section \ref{Section:Sec1} we formulate the problem set-up and introduce two different tamed Milstein schemes, in which we use two different taming factors for the drift. The proofs of the main convergence results are deferred to Section \ref{Section:Sec3}. Section \ref{Section:Sec4} illustrates the numerical performance of the proposed time-stepping schemes and reveals that the schemes can also successfully be applied to equations which do not satisfy all imposed model assumptions. Appendix \ref{Appendx:A} introduces a notion which allows to consider derivatives on the Wasserstein space. We end this section by fixing the notation and by introducing several notions needed throughout this paper. \\ \\
\noindent 
\textbf{Preliminaries:} \\ \\
Let $(\R^d,\<\cdot,\cdot\>,|\cdot|)$ represent the $d$-dimensional
Euclidean space and $\R^d\otimes\R^m$ be the collection of all $d\times m$-matrices. The transpose of a matrix $A$ will be denoted by $A^{*}$. In addition, we use $\mathscr{P}(\R^d)$ to denote the family of all probability
measures on $(\R^d,\mathcal{B}(\R^{d}))$, where $\mathcal{B}(\R^{d})$ denotes the Borel $\sigma$-field over $\R^d$, and define the subset of probability measures with finite second moment by
\begin{equation*}
\mathscr{P}_2(\R^d):= \Big \{ \mu\in\mathscr
{P}(\R^d) \Big| \ \int_{\R^d}|x|^2\mu(\d x)<\8 \Big \}.
\end{equation*}
For all linear (e.g., matrices), and bilinear operators appearing in this article, we will use the standard Hilbert-Schmidt norm denoted by $\| \cdot \|$.

As metric on the space $\mathscr{P}_2(\R^d)$, we use the Wasserstein distance. For $\mu,\nu\in\mathscr{P}_2(\R^d)$, the Wasserstein distance between $\mu$ and $\nu$ is defined as
\begin{equation*}
\mathbb{W}_2(\mu,\nu) := \inf_{\pi\in\mathcal
{C}(\mu,\nu)} \left( \int_{\R^d\times \R^d}|x-y|^2\pi(\d x,\d y) \right)^{1/2},
\end{equation*}
where $\mathcal {C}(\mu,\nu)$ is the set of all couplings of $\mu$ and
$\nu$, i.e., $\pi\in\mathcal {C}(\mu,\nu)$ if and only if
$\pi(\cdot,\mathbb{R}^d)=\mu(\cdot)$ and $\pi(\mathbb{R}^d,\cdot)=\nu(\cdot)$. Let $(\OO,\F,(\F_t)_{t\ge0},\P)$ be a filtered probability space satisfying the usual assumptions. For a given $p \geq 1$, $L_p^{0}(\mathbb{R}^d)$ will denote the space of $\mathbb{R}^d$-valued, $\F_0$-measurable random variables $X$ satisfying $\mathbb{E}|X|^p < \infty$. Further, $\mathcal{S}^p([0,T])$ refers to the space of $\mathbb{R}^d$-valued continuous, $\F$-adapted processes, defined on the interval $[0,T]$, with finite $p$-th moments (uniform in time).
\section{Tamed Milstein schemes for non-Lipschitz McKean--Vlasov SDEs}\label{Section:Sec1}

In this section, we define the model set-up with its assumptions (subsection \ref{subsec:prelims}), and define the Milstein schemes (subsection \ref{subsec:schemes}). 
We focus in the analysis on the schemes with tamed drift
coefficients, where the super-linear drifts are approximated by functions which are bounded depending on the mesh size (similar to the tamed Euler schemes in \cite{RES}).
A simplified version of the proofs gives the corresponding results for the Milstein scheme without drift approximation in the global Lipschitz case, and we only state the
assumptions required in subsection \ref{subsec:lipschitz}.

\subsection{Assumptions and interacting particle system}
\label{subsec:prelims}
For a given time horizon $[0,T]$ with terminal time $T>0$, we consider the following McKean--Vlasov SDE on $\R^d$,
\begin{equation}\label{EE1}
\d X_t=b \left(X_t,\mathscr{L}_{X_t} \right)\d t+\si(X_t,\mathscr{L}_{X_t})\d W_t, \quad X_0=\xi,
\end{equation}
where we recall that $\mathscr{L}_{X_t}$ denotes the marginal law of $X$ at the time $t \geq 0$, 
$b:\R^d\times \mathscr {P}_2(\R^d) \to \R^{d}$, $\si: \R^d\times \mathscr
{P}_2(\R^d)\to \R^{d} \otimes \R^{m}$ are measurable functions, $\xi \in L_p^{0}(\mathbb{R}^d)$, for all $p \geq 1$, and $(W_t)_{t\ge0}$ is an
$m$-dimensional Brownian motion on the filtered, atomless probability space
$(\OO,\F,(\F_t)_{t\ge0},\P)$, where $(\F_t)_{t\ge0}$ is the natural filtration of $(W_t)_{t \geq 0}$ augmented with an independent $\sigma$-algebra $\F_0$. 

We assume, for any $x,y\in\R^d $ and $\mu,\nu\in\mathscr
{P}_2(\R^d)$ the following:
\begin{enumerate}
\item[({\bf A}$_b^1$)] There exist constants $L_b^1,\alpha_1>0$ such that
\begin{align*}
&\<x-y,b(x,\mu)-b(y,\mu)\>\le L_b^1 |x-y|^2, \tag{1} \\
&|b(x,\mu)-b(y,\mu)|\le L_b^1(1+|x|^{\aa_1}+|y|^{\aa_1})|x-y|, \tag{2} \\
&|b(x,\mu)-b(x,\nu)|\le L_b^1\mathbb{W}_2(\mu,\nu) \tag{3}.
\end{align*}
\item[({\bf A}$_\si^1$)] There exists a constant $L_\si^1>0$ such that
\begin{equation*}
\|\si(x,\mu)-\si(y,\nu)\| \le L_\si^1(|x-y|+\mathbb{W}_2(\mu,\nu)).
\end{equation*}
\end{enumerate}

The analysis presented below can be readily extended to the case where $b$ and $\sigma$ depend explicitly on $t$ in a Lipschitz continuous way.

We recall that under these assumptions \cite[Theorem 3.3]{RST} guarantees that \eqref{EE1} has a unique strong
solution with bounded $p$-th moments, i.e., we have
\begin{equation}\label{YY1}
\E\|X\|^p_{\8,T}<\8, 
\end{equation}
where $\|X\|_{\8,t}^p:= \sup_{0 \leq s \leq t} |X_s|^p$, for $t \geq 0$ and $p \geq 1$.


For each $i\in \mathbb{S}_N:=\{1,\ldots,N\}$, let $(W^i,X_0^i)$ be
independent copies of $(W,X_0).$ Consider first the following
{\it non-interacting} particle system associated with \eqref{EE1},
\begin{equation}\label{E2}
\d X_t^i=b(X_t^i,\mathscr{L}_{X_t^{i}})\d t+\si(X_t^i,\mathscr{L}_{X_t^{i}})\d W_t^i, \quad \mathscr{L}_{X_0^i}=\mathscr{L}_{X_0},~~i\in \mathbb{S}_N.
\end{equation}
One obviously has $\mathscr{L}_{X_t}=\mathscr{L}_{X_t^i}$, $i\in\mathbb{S}_N$. Compared to the simulation of classical SDEs, the key difference is the need to approximate the measure $\mathscr{L}_{X_t}$, for each $t \geq 0$. To do so, we introduce
the following {\it interacting} particle system (see e.g., \cite{BT})
\begin{equation}\label{E3}
\d X_t^{i,N}=b(X_t^{i,N},\mu_t^{X^{\cdot,N} })\d
t+\si(X_t^{i,N},\mu_t^{X^{\cdot,N} })\d W_t^i,
~~X_0^{i,N}=X_0^{i},~~i\in \mathbb{S}_N,
\end{equation}
where $\mu_t^{X^{\cdot,N} }(\mathrm{d}x)
:=\ff{1}{N}\sum_{j=1}^N\dd_{X_t^{j,N}}(\mathrm{d}x)$. Set
\begin{equation*}
B({\bf x}):=(b(x_1,\hat\mu^{{\bf x},N}),\ldots,b(x_N,\hat\mu^{{\bf
x}, N}))^*,~~~\Sigma({\bf x}):=\mbox{diag}(\si(x_1,\hat\mu^{{\bf
x},N}),\ldots,\si(x_N,\hat\mu^{{\bf x},N}))
\end{equation*}
where $\hat\mu^{{\bf x},N}(\mathrm{d}x):=\ff{1}{N}\sum_{j=1}^N\dd_{x_j}(\mathrm{d}x)$ for
${\bf x}:=(x_1,x_2,\ldots,x_N)$. Note that
\begin{equation}\label{D4}
\mathbb{W}_2(\hat\mu^{{\bf x},N},\hat\mu^{{\bf y},N})^2\le
\ff{1}{N}|{\bf x}-{\bf y}|^2,~~~{\bf x},{\bf y}\in \R^{dN}.
\end{equation}
Hence, we deduce from ({\bf A}$_b^1$) and ({\bf A}$_\sigma^1$) that there exists a
constant $L>0$ such that
\begin{equation*}
\<{\bf x}-{\bf y}, B({\bf x})-B({\bf y})\>\le L|{\bf x}-{\bf
y}|^2,~~~\|\Sigma({\bf x})-\Sigma({\bf y})\|^2 \le
L|{\bf x}-{\bf y}|^2
\end{equation*}
for any ${\bf x},{\bf y}\in \R^{dN}$. Consequently, according to, e.g., \cite[Theorem 3.1.1]{PR}, the stochastic interacting particle
system \eqref{E3} is well-posed.

\subsection{Time-stepping schemes and main results}
\label{subsec:schemes}

Since the drift term $b$ is not assumed to be Lipschitz with respect to the spatial argument, the standard Euler scheme is, in general, not suitable for \eqref{E3}, due to the potential moment-unboundedness of the time-discretised interacting particle system \cite{HJK2}. Here, we propose two novel stable time-stepping schemes which achieve a first order strong convergence rate.

We partition a given time interval $[0,T]$ into $M \in \mathbb{N}$ steps of equal length $\delta:=T/M$. In the sequel, we set $t_n:=n \delta$ and for any $t \in [0,T]$, we define $t_\dd := \max \lbrace t_n| \  t_n \leq t  \rbrace =  \lfloor t / \dd \rfloor \delta$. Now, for $n \in \lbrace 0, \ldots, M-1 \rbrace$, we introduce the following tamed Milstein schemes: For each $i \in \mathbb{S}_N$, $Y_{t_{n+1}}^{i,N}$ is computed by 
\begin{align}\label{eq:MilsteinDiscrete}
Y_{t_{n+1}}^{i,N} &= Y_{t_n}^{i,N} + b_\dd(Y_{t_n}^{i,N},\mu_{t_n}^{Y^{\cdot,N}}) \, \delta +  \sigma(Y_{t_n}^{i,N},\mu_{t_n}^{Y^{\cdot,N}}) \Delta W_n^{i} \nonumber \\
& \quad+ \int_{t_n}^{t_{n+1}} \nn \sigma(Y_{t_n}^{i,N},\mu_{t_n}^{Y^{\cdot,N}}) \sigma(Y_{t_n}^{i,N}, \mu_{t_n}^{Y^{\cdot,N}})  \int_{t_n}^{s} \mathrm{d}W^{i}_u  \mathrm{d}W^{i}_s  \nonumber \\
& \quad + \int_{t_n}^{t_{n+1}} \frac{1}{N}\sum_{j = 1}^N  D^{L} \sigma(Y_{t_n}^{i,N}, \mu_{t_n}^{Y^{\cdot,N}})(Y_{t_n}^{j,N}) \sigma(Y_{t_n}^{j,N}, \mu_{t_n}^{Y^{\cdot,N}}) \int_{t_n}^{s} \mathrm{d}W^{j}_u  \mathrm{d}W^{i}_s,
\end{align}
where the driving Brownian motions $W^{i}$ and initial values $Y_{0}^{i,N} = X_0^{i,N}$ are assumed to be independent and $\Delta W_n^{i} := W^{i}_{t_{n+1}} - W^{i}_{t_n}$. In addition, we used the notation
\begin{equation*}
\mu_{t_n}^{Y^{\cdot,N}}(\mathrm{d}x) := \frac{1}{N} \sum_{j=1}^{N} \delta_{Y_{t_n}^{j,N}}(\mathrm{d}x).
\end{equation*}
Above, $\nn \si$ denotes the first order gradient operator (applied to each column of $\sigma$) with respect to the state variable of $\si$ and $D^{L} \sigma$ is the Lions derivative operator (see Appendix \ref{Appendx:A} for details). Note that for $x,y \in \mathbb{R}^d$, $\mu \in \mathscr{P}_2(\R^{d})$, $\nn\si(x,\mu)$ and $D^L\si(x,\mu)(y)$ are tensors and can be viewed as linear operators from $\R^d$ to $\R^d\otimes\R^m$.

We define $b_\dd$ in two different ways yielding two schemes, which will subsequently be denoted by Scheme $1$ and Scheme $2$, respectively: 
For Scheme 1, we use  
\begin{equation*}
b_\dd(x,\mu):=\ff{b(x,\mu)}{1+\dd|b(x,\mu)|}, \quad x\in\R^d, \
\mu\in\mathscr {P}_2(\R^d),
\end{equation*}
and for Scheme 2, we define  
\begin{equation*}
b_\dd(x,\mu):=\ff{b(x,\mu)}{1+\dd|b(x,\mu)|^2}, \quad x\in\R^d, \
\mu\in\mathscr {P}_2(\R^d).
\end{equation*}
Note that the following bounds hold for the two different choices of $b_\dd$:
\begin{align}\label{eq:boundsdrift}
\ff{|b(x,\mu)|}{1+\dd|b(x,\mu)|} \leq |b(x,\mu)| \land \frac{M}{T}, \qquad \quad
\ff{|b(x,\mu)|}{1+\dd|b(x,\mu)|^2} \leq |b(x,\mu)| \land  \sqrt{\frac{M}{T}}. 
\end{align}

\begin{rmk}
If the drift $b$ is also globally Lipschitz in the state component, taming is not necessary and we can instead introduce a standard Milstein scheme by simply replacing $b_\dd$ in (\ref{eq:MilsteinDiscrete}) with $b$; see subsection \ref{subsec:lipschitz} for details of this case.
\end{rmk}

The continuous time version of (\ref{eq:MilsteinDiscrete}) reads
\begin{equation}\label{eq:ContinuousMilstein}
\begin{split}
\d Y_t^{i,N}&=b_\dd(Y_{t_\dd}^{i,N},\mu_{t_\dd}^{Y^{\cdot,N}})\d
t+\bigg(\si(Y_{t_\dd}^{i,N},\mu_{t_\dd}^{Y^{\cdot,N}})+\int_{t_\dd}^t(\nn\si)\si(Y_{r_\dd}^{i,N},\mu_{r_\dd}^{Y^{\cdot,N}})\d
W_r^i\\
&\quad+\ff{1}{N}\sum_{j=1}^N\int_{t_\dd}^t D^L\si(Y_{r_\dd}^{i,N},\mu_{r_\dd}^{Y^{\cdot,N}})(Y_{r_\dd}^{j,N})\si(Y_{r_\dd}^{j,N},
\mu_{r_\dd}^{Y^{\cdot,N}})\d W_r^j\bigg)\d W_t^i,
\end{split}
\end{equation}
with $Y_0^{i,N}=X_0^{i,N}$.

\begin{rmk}
Compared to the classical Milstein scheme for standard SDEs without measure dependence, a term involving the $L$-derivative appears. Although this term is crucial for the theoretical analysis of the scheme for fixed $N$, it can be dropped in practice for large $N$ and when an approximation to the limiting McKean--Vlasov equation is sought.
We will provide a theoretical justification in the Lipschitz case at the end of subsection \ref{subsec:lipschitz}, and give a numerical illustration in Section \ref{Section:Sec4}.
\end{rmk}

In what follows, we show that the fully discretised particle system converges, in a strong sense, to a solution of the limit McKean--Vlasov SDE if $N \to \infty$ and $\delta \to 0$ and establish the order of convergence. To establish the following main result on moment stability and strong convergence of the above proposed time-stepping schemes, we need further assumptions on the coefficients $b$ and $\sigma$. We refer the reader to Appendix \ref{Appendx:A} for the precise definitions of the function spaces used in the assumptions listed below, in particular the class $ C^{2,(2,1)}(\R^{d}\times \mathscr{P}_2(\R^{d}))$.

For any $x, x', y, y' \in \R^d$ and $\mu,\nu \in \mathscr
{P}_2(\R^d)$, we require the following:
\begin{enumerate}
\item[({\bf A}$_b^2$)] Set $b=(b_1,\ldots,b_d)^{*}$, let $b_i \in C^{2,(2,1)}(\R^{d}\times \mathscr{P}_2(\R^{d}))$ and assume that there exist constants $L_b^2,\aa_2>0$ such that for all $i \in \lbrace 1, \ldots, d \rbrace$
\begin{equation*}
\begin{split}
&\|\nn^2b_i(x,\mu)\|  \vee\|\nn
\{D^Lb_i(x,\mu)(\cdot)\}(y)\| \vee\| (D^L)^2b_i(x,\mu)(y,y)\|\\
&\quad~~~~~~~~~~~~~~~~~~~~~~~~~~~~~~~~~~~~~~~~~~~~~~~~~~~~~~~~~  \le L_b^2(1+|x|^{1+\aa_2}+|y|^{1+\aa_2}+\mu(|\cdot|^2)^{\frac{1+\aa_2}{2}}),
\end{split}
\end{equation*}
where $\nn^2$ is the second order gradient operator with respect to the first
argument and $(D^L)^2$ the second order $L$-derivative operator.
\item[({\bf A}$_b^3$)] There exist constants $ L_b^3,\aa_3>0$ such that
\begin{equation*}
\begin{split}
&\|\nn b(x,\mu)-\nn b(y,\nu)\| \le
L_b^3(|x-y|+\mathbb{W}_2(\mu,\nu))\\
&\quad~~~~~~~~~~~~~~~~~~~~~~~~~~~~~~~~~~~~~ \times(1+|x|^{\aa_3}+|y|^{\aa_3}+\mu(|\cdot|^2)^{\frac{\aa_3}{2}}+\nu(|\cdot|^2)^{\frac{\aa_3}{2}}), \\
&\|D^Lb(x,\mu)(y)-D^Lb(x',\nu)(y')\| \le
L_b^3(|x-x'|+|y-y'|+\mathbb{W}_2(\mu,\nu))\\
&\quad~~~~~~~~~~~~~~~~~~~~~~~~~~~~~~~~~
\times(1+|x|^{\aa_3}+|y|^{\aa_3}+|x'|^{\aa_3}+|y'|^{\aa_3}+\mu(|\cdot|^2)^{\frac{\aa_3}{2}}+\nu(|\cdot|^2)^{\frac{\aa_3}{2}}).
\end{split}
\end{equation*}
\item[({\bf A}$_b^4$)] There exists a constant $ L_b^4 >0$ such that
$
| b(0,\mu) | \leq L_b^4.
$
\end{enumerate}

Concerning the diffusion coefficient $\si$, we further impose, for all
$x,x',y,y'\in\R^d$ and $\mu,\nu\in\mathscr{P}_2(\R^d)$:
\begin{enumerate}
\item[({\bf A}$_\si^2$)] Let $\sigma_{ij} \in C^{2,(2,1)}(\R^{d}\times \mathscr{P}_2(\R^{d}))$, where $\si_{ij}$ denotes the $(i,j)$-th component of $\sigma$, for $i \in \lbrace 1, \ldots, m \rbrace$, $j \in \lbrace 1, \ldots, d \rbrace$ and assume that there exists a constant $L_\si^2>0$ such that for all $i \in \lbrace 1, \ldots, m \rbrace$, $j \in \lbrace 1, \ldots, d \rbrace$
\begin{equation*}
\begin{split}
& \| D^{L} \si_{ij}(x,\mu)(y) \| \vee \|\nn
\{D^L\si_{ij}(x,\mu)(\cdot)\}(y)\| \vee \| (D^L)^2\si_{ij}(x,\mu)(y,y)\|   
\le L_\si^2.
\end{split}
\end{equation*}
\item[({\bf A}$_\si^3$)] There exists a constant $L_\si^3 >0$ such that
\begin{align*}
&\|\nn\si(x,\mu)-\nn\si(y,\nu)\|\le
 L_\si^3(|x-y|+\mathbb{W}_2(\mu,\nu)) \tag{1}, \\
&\|D^L\si(x,\mu)(y)-D^L\si(x',\nu)(y')\|\le
L_\si^3(|x-x'|+|y-y'|+\mathbb{W}_2(\mu,\nu)). \tag{2}
\end{align*}
\item[({\bf A}$_\si^4$)] There exists a constant $L_{\si}^4 >0$ such that
\begin{align*}
& \|\sigma(0,\mu) \| +\| D^L \si(x,\mu)(y) \si(y,\mu) \|\leq L_\si^4.
\end{align*}
\end{enumerate}
\begin{rmk}
The first inequality in ({\bf A}$^1_b$) is the so called one-sided Lipschitz condition which is needed to control the polynomial growth of the drift in the state variable (uniformly with respect to the measure variable). The third inequality in ({\bf A}$^1_b$) and assumption ({\bf A}$^1_\sigma$) express that both coefficients are globally Lipschitz continuous in the measure component (uniformly in the state variable). In ({\bf A}$^2_b$)-({\bf A}$^3_b$) and ({\bf A}$^2_\sigma$)-({\bf A}$^3_\sigma$), we require growth and Lipschitz conditions on the derivatives of $b$ and $\sigma$, which are necessary for the strong convergence analysis. Further, we assume uniform boundedness of the coefficients in the measure component in  ({\bf A}$^4_b$) and ({\bf A}$^4_\sigma$), which is essential to achieve moment boundedness of Scheme 1.   
\end{rmk}
\noindent
Now we are in a position to present our first main result (concerned with Scheme 1) and we remark that in Section \ref{Section:Sec3} we give more details on the generic constants used in the statement of the results.

\begin{lem}\label{TH:TH1}
{\rm
Let $Y^{i,N}_{t_n}$, $n \in \lbrace 0, \ldots, M \rbrace$, be defined as in (\ref{eq:MilsteinDiscrete}) with $b_\dd(x,\mu)=\ff{b(x,\mu)}{1+\dd|b(x,\mu)|}$ and $p \geq 1$. Then, under assumptions ({\bf A}$_b^1$), ({\bf A}$_b^4$), ({\bf A}$_\si^1$), ({\bf A}$_\si^2$) and ({\bf A}$_\si^4$), there exists a constant $C >0$ independent of $M$ ($\delta$, respectively) such that 
\begin{equation*} 
    \sup_{i \in \mathbb{S}_N} \sup_{n \in \lbrace 0, \ldots, M \rbrace}  \mathbb{E} | Y^{i,N}_{t_n} |^p  \leq C. 
\end{equation*} 
}
\end{lem}
\begin{proof}
The proof is deferred to Section \ref{Section:Sec3}, subsection \ref{subsec:stab_lem}. 
\end{proof}

\begin{thm}\label{TH2:TH2}
{\rm Let $p \geq 1$. Assume ({\bf A}$_b^1$)--({\bf A}$_b^4$), and ({\bf
A}$_\si^1$)--({\bf A}$_\si^4$). Let $(Y^{i,N}_t)_{t \in [0,T]}$ be defined by (\ref{eq:ContinuousMilstein}) with $b_\dd(x,\mu)=\ff{b(x,\mu)}{1+\dd|b(x,\mu)|}$. Then there exists a constant $C>0$ independent of $M$ ($\delta$, respectively) and $N$ such
that
\begin{equation}\label{D2}
\E\|X^{i,N}-Y^{i,N}\|^p_{\8, T}\le C \dd^p, \qquad i\in\mathbb{S}_N.
\end{equation}
}
\end{thm}
\begin{proof}
The proof is deferred to Section \ref{Section:Sec3}. 
\end{proof}

The following corollary is an immediate consequence of the above theorem and the pathwise propagation of chaos result in \cite[Proposition 3.1]{RES}.
\begin{cor}\label{Coro1}
Let the assumptions of Theorem \ref{TH2:TH2} hold for $p=2$.
Then there exists a constant $C>0$ independent of $M$ ($\delta$, respectively) and $N$ such that
\begin{equation*}
\E\|X^{i}-Y^{i,N}\|^2_{\8, T}\le C (\dd^2+\phi(N)), \qquad i\in\mathbb{S}_N,
\end{equation*}
where
\begin{equation}\label{D1}
\phi(N)=
\begin{cases}
N^{-1/2}, &\text{ for }  d<4, \\
N^{-1/2}\log N, &\text{ for }  d=4, \\
N^{-2/d}, &\text{ for } d>4.
\end{cases}
\end{equation}
\end{cor}
\begin{proof}
We have for some constant $C>0$ (independent of $M$ and $N$)
\begin{equation*}
|X^i_t-Y^{i,N}_t|^2 \leq C(|X^i_t-X^{i,N}_t|^2+|X^{i,N}_t - Y^{i,N}_t|^2).
\end{equation*}
The $\mathcal{S}^2$-norm of the second summand is of order $\delta^2$ due to Theorem \ref{TH2:TH2},
and from \cite[Proposition 3.1]{RES},
there exists a constant $C>0$ (independent of $N$) such that
\begin{equation}\label{E4}
\sup_{i\in \mathbb{S}_N}\E\|X^i-X^{i,N}\|_{\8,T}^2 \leq C \phi(N),
\end{equation}
such that the claim follows.
\end{proof}
\noindent

\begin{rmk}\label{rmk:TH2Scheme2}
Let $(Y^{i,N}_t)_{t \in [0,T]}$ be defined by Scheme 2, i.e., (\ref{eq:ContinuousMilstein}) with $b_\dd(x,\mu) = b(x,\mu)/(1+\dd|b(x,\mu)|^2)$.
Assume, for some $p\ge 1$, ({\bf A}$_b^1$)--({\bf A}$_b^3$), ({\bf
A}$_\si^1$)--({\bf A}$_\si^3$) and $X_0 \in L_{4p(1+\alpha)}^{0}(\mathbb{R}^d)$, where $\alpha = \alpha_1 \lor \alpha_2 \lor \alpha_3$. 
Then the statement of Theorem \ref{TH2:TH2} still holds.

Due to the choice of taming, we have the stronger bound in (\ref{eq:boundsdrift}) and the moment boundedness of Scheme 2 follows immediately from \cite[Lemma 4.3]{stock1}
(see also Remark \ref{rem:taming} below).
The proof of the strong convergence rate is then analogous to Scheme 1 and is therefore omitted. 

\label{rem:taming}
The crucial difference between these two results is that for Scheme 2 
we do not need to require 
({\bf A}$_b^4$) and ({\bf A}$_\si^4$), i.e.,
that the coefficients are uniformly bounded in the measure component. The reason is that the taming for Scheme 2 is stronger than for Scheme 1. The numerical tests in Section \ref{Section:Sec4}  show, however, that Scheme 1 is significantly more accurate in all cases studied, even when 
({\bf A}$_b^4$) and ({\bf A}$_\si^4$) are violated.
\end{rmk}

\begin{rmk}
To prove moment boundedness for Scheme 1, a pathwise estimate of the time-discretised particle system is required. In order to obtain such an estimate the mean-field terms need to be controlled. 
 Here, our main contribution is to show how the L\'{e}vy area terms in above schemes can be handled in a pathwise sense (see Lemma \ref{TH:TH1}). In \cite{GW}, a commutativity assumption for the diffusion matrix is imposed, such that the L\'{e}vy areas vanish. The techniques used to prove our result can also be employed to relax the assumptions imposed in \cite{GW}.
\end{rmk}

\subsection{Milstein scheme for globally Lipschitz drift}
\label{subsec:lipschitz}

For the sake of completeness, we give a set of model assumptions which allow to derive an analogous convergence result for a standard Milstein scheme, i.e., without taming the drift.
We assume ({\bf A}$_\si^1$)--({\bf A}$_\si^3$) for the diffusion coefficient, and for any $x,x',y,y' \in\R^d $ and $\mu,\nu\in\mathscr {P}_2(\R^d)$ we impose:
\begin{enumerate}
\item[({\bf AA}$_b^1$)] There exists a constant $L_b^1>0$ such that
\begin{align*}
&\|b(x,\mu)-b(y,\nu)\| \le L_b^1(|x-y|+\mathbb{W}_2(\mu,\nu)).
\end{align*}
\item[({\bf AA}$_b^2$)] Set $b=(b_1,\ldots,b_d)^{*}$, let $b_i \in C^{2,(2,1)}(\R^{d}\times \mathscr{P}_2(\R^{d}))$ and assume that there exists a constant $L_b^2>0$ such that for all $i \in \lbrace 1, \ldots, d \rbrace$
\begin{equation*}
\begin{split}
& \|\nn
\{D^L b_i(x,\mu)(\cdot)\}(y)\| \vee \| (D^L)^2 b_i(x,\mu)(y,y)\|   
\le L_b^2.
\end{split}
\end{equation*}
\item[({\bf AA}$_b^3$)] There exists a constant $L_b^3 >0$ such that
\begin{align*}
&\|\nn b(x,\mu)-\nn b(y,\nu)\|\le
 L_b^3(|x-y|+\mathbb{W}_2(\mu,\nu)) \tag{1}, \\
&\|D^L b(x,\mu)(y)-D^L b(x',\nu)(y')\|\le
L_b^3(|x-x'|+|y-y'|+\mathbb{W}_2(\mu,\nu)). \tag{2}
\end{align*}
\end{enumerate}

Under the assumptions listed above, the standard Milstein scheme defined as in (\ref{eq:MilsteinDiscrete}), with $b_\dd$ replaced by $b$, is stable (i.e., has bounded moments) and converges with strong order 1.
As $b$ and $\sigma$ have linear growth in both components and the particles are identically distributed, the claim concerning the stability of the scheme follows by a standard Gronwall type argument.
The proof for the strong convergence order is a simplified version of the proof of Theorem \ref{TH2:TH2} and is therefore omitted. 

We end by estimating the term involving the $L$-derivative in (\ref{eq:MilsteinDiscrete}), 
\begin{equation}\label{eq:LionsTerm}
\int_{t_n}^{t_{n+1}} \ff{1}{N}\sum_{j=1}^N
\int_{t_\dd}^tD^L\si(Y_{r_\dd}^{i,N},\mu_{r_\dd}^{Y^{\cdot,N}})(Y_{r_\dd}^{j,N})\si(Y_{r_\dd}^{j,N},
\mu_{r_\dd}^{Y^{\cdot,N}})\d W_r^j \d W_t^i,
\end{equation}
which is expected to be close to zero for large $N$, as explained by the following heuristic observation.

Instead of discretising the particle system, we can directly discretise (\ref{EE1}) in time using a Milstein scheme without employing the particle approximation for the measure.
It\^{o}'s formula applied to a function from $C^{2,(1,1)}(\R^{d}\times \mathscr{P}_2(\R^{d}))$, in our case to the components $\si_{ij}$ of the diffusion coefficient, $i \in \lbrace 1, \ldots, m \rbrace$, $j \in \lbrace 1, \ldots, d \rbrace$, gives 
\begin{align*}
\mathrm{d}\sigma_{ij}(X_s,\mathscr{L}_{X_s}) &= \Bigg[\frac{1}{2} \textrm{tr}(\sigma \sigma^{*} \nabla^{2} \sigma_{ij})(X_s,\mathscr{L}_{X_s}) + \left \langle b, \partial_x \sigma_{ij} \right \rangle (X_s,\mathscr{L}_{X_s}) \\
& \hspace{0.5cm} + \int_{\mathbb{R}^d} \Big[ \frac{1}{2} \textrm{tr} \Big \lbrace \sigma \sigma^{*}(y,\mathscr{L}_{X_s}) \nabla \lbrace D^{L} \sigma_{ij}(X_s,\mathscr{L}_{X_s})(\cdot)\rbrace(y)\Big \rbrace  \\
&\hspace{0.5cm} +\left \langle b(y,\mathscr{L}_{X_s}), D^{L} \sigma_{ij}(X_s,\mathscr{L}_{X_s})(y) \right \rangle \Big] \mathscr{L}_{X_s}(\mathrm{d}y) \Bigg] \mathrm{d}s \\
&\hspace{0.5cm} + \left \langle (\sigma^{*} \nabla \sigma_{ij})(X_s,\mathscr{L}_{X_s}), \mathrm{d}W_s \right \rangle, 
\end{align*}   
where $\textrm{tr}$ denotes the standard trace operator (see \cite[Proposition 5.102]{CD}). Using this expansion, we can introduce a Milstein scheme for (\ref{EE1}) 
\begin{equation}\label{eq:MStemp}
Y_{t_{n+1}} = Y_{t_n} + b(Y_{t_n}, \mathscr{L}_{Y_{t_n}}) \, \delta + \sigma(Y_{t_n}, \mathscr{L}_{Y_{t_n}}) \Delta W_n + \int_{t_n}^{t_{n+1}} \nn \sigma(Y_{t_n}, \mathscr{L}_{Y_{t_n}}) \sigma(Y_{t_n}, \mathscr{L}_{Y_{t_n}})  \int_{t_n}^{s} \mathrm{d}W_u  \mathrm{d}W_s,
\end{equation}
for $n \in \lbrace 0, \ldots, M-1 \rbrace$, with $Y_0 = X_0$. Therefore, it is reasonable to expect that the term (\ref{eq:LionsTerm}) vanishes as $N \to \infty$, since (\ref{eq:MStemp}) and (\ref{eq:MilsteinDiscrete}) (with $b_\dd$ replaced by $b$, as we restrict this discussion to a global Lipschitz setting) should coincide for $N \to \infty$,
i.e., a propagation of chaos result on the level of the time-discrete system.

More precisely, we give the following proposition. 
We only prove this statement for globally Lipschitz coefficients, but expect a similar result to hold in the general setting of this paper. To prove the claim we additionally require the following for any $x,x',y,y' \in \R^{d}$ and $\mu,\nu \in \mathscr{P}_2(\R^{d})$: 
\begin{itemize}
\item[({\bf AA}$_{\sigma}^1$)] Let $\sigma$ be continuously differentiable in both components and assume that there exists a constant $L_{\sigma}^1 >0$ such that
\begin{align*}
&\|(\nn \sigma) \sigma(x,\mu)-(\nn \sigma) \sigma(x',\nu)\|\le
 L_{\sigma}^1(|x-x'|+\mathbb{W}_2(\mu,\nu)),  \tag{1} \\
&\|D^L \sigma(x,\mu)(y)\sigma(y,\mu)-D^L \sigma(x',\nu)(y')\sigma(y',\nu)\|\le
L_{\sigma}^1(|x-x'|+|y-y'|+\mathbb{W}_2(\mu,\nu)).  \tag{2} 
\end{align*}
\end{itemize}
\begin{prp}\label{lemma:POCTime}
\rm{Assume ({\bf AA}$_b^1$), ({\bf A}$_\si^1$), ({\bf AA}$_\si^1$) and $X_0 \in L_{2}^{0}(\mathbb{R}^d)$. Further, for $i\in \mathbb{S}_N$ and $n \in \lbrace 0, \ldots, M \rbrace$, let $Y^{i}_{t_n}$ be given by (\ref{eq:MStemp}), 
with independent copies $(W^{i},X_0^{i})$ of $(W,X_0)$,
and $Y^{i,N}_{t_n}$ be defined by (\ref{eq:MilsteinDiscrete}), with $b_\dd$ replaced by $b$. Then, there exists a constant $C>0$ such that 
\begin{equation*}
\sup_{i \in \mathbb{S}_N} \sup_{n \in \lbrace 0, \ldots, M \rbrace} \mathbb{E} \left|  Y_{t_{n}}^{i,N} - Y_{t_{n}}^{i} \right|^2 \leq C\phi(N).
\end{equation*}
\rm}
\end{prp}
\begin{proof}
See subsection \ref{Sec:POCTime}. 
\end{proof}

\section{Proofs of results}\label{Section:Sec3}

Here, and throughout the remaining article, we write $a \1 b$ to express that there exists a constant $C>0$ such that $a \leq Cb$, where $a,b \in \mathbb{R}$. The implied constant $C>0$ may depend on the parameters $p,\vv,T, m, d$, the constants appearing in above assumptions and the moments of the initial data, but is independent of $M$ (and $\delta$, respectively) and $N$. Also implied constants may change their values from line to line in a sequence of inequalities. To make the presentation clearer, we split the proof of Theorem \ref{TH2:TH2} into several auxiliary lemmata.

For the sake of readability, we also set, for any $i \in\mathbb{S}_N$,
\begin{equation}\label{D3}
\begin{split}
\Upsilon_t^i:&=
\si(Y_{t_\dd}^{i,N},\mu_{t_\dd}^{Y^{\cdot,N}})+\int_{t_\dd}^t\nn\si(\cdot,\mu_{r_\dd}^{Y^{\cdot,N}})(Y_{r_\dd}^{i,N})\si(Y_{r_\dd}^{i,N},\mu_{r_\dd}^{Y^{\cdot,N}})\d
W_r^i\\
&\quad~~~~~~~~~~~~~~~~~~~~+\ff{1}{N}\sum_{j=1}^N\int_{t_\dd}^tD^L\si(Y_{r_\dd}^{i,N},\mu_{r_\dd}^{Y^{\cdot,N}})(Y_{r_\dd}^{j,N})\si(Y_{r_\dd}^{j,N},
\mu_{r_\dd}^{Y^{\cdot,N}})\d W_r^j, \\
 \Gamma_t^i:&=
\si(Y_t^{i,N},\mu_t^{Y^{\cdot,N} })-\Upsilon_t^i, \\
M_t^i:&=\int_0^t\{\si(X_s^{i,N},\mu_s^{X^{\cdot,N}
})-\Upsilon_s^i\}\d W_s^i, \\
Z^{i,N}:&=X^ {i,N}-Y^{i,N}.
\end{split}
\end{equation}

\subsection{Proof of Lemma \ref{TH:TH1}}
\label{subsec:stab_lem}

\begin{proof}
We utilise ideas developed in \cite{HJK} for tamed schemes for standard SDEs. We will point these out and focus on the key differences to \cite{HJK}.

For $i \in \mathbb{S}_N$, we set 
\begin{align*}
A_n^{i} :=  \sum_{r,q=1}^{m} \left|\int_{t_n}^{t_{n+1}} \Delta W_s^{i,r} \d W_s^{i,q} \right|, \qquad\quad
A_n^{i,N} := \frac{1}{N} \sum_{j=1}^{N} \sum_{r,q=1}^{m} \left|\int_{t_n}^{t_{n+1}} \Delta W_s^{j,r} \d W_s^{i,q} \right|,
\end{align*}
where $W_s^{i,r}$ denotes the $p$-th component of the Brownian motion $W^{i}$ and $\Delta W_s^{i} = W_s^{i} - W_{t_n}^i$. \\
Define
\begin{align*} 
& \Omega^{i}_n := \left \lbrace \omega \in \Omega \ \Big| \ \sup_{k \in \lbrace 0, \ldots, n-1 \rbrace} A_k^{i}(\omega) \leq \frac{1}{\sqrt{M}}, \sup_{k \in \lbrace 0, \ldots, n-1 \rbrace} A_k^{i,N}(\omega) \leq \frac{1}{\sqrt{M}} \right \rbrace, \\ 
& \Omega^{i}_{n, \lambda} := \left \lbrace \omega \in \Omega \ \Big| \ \sup_{k \in \lbrace 0, \ldots, n-1 \rbrace} \left( D_{k, \lambda}^{i}(\omega) + D_{k, \lambda}(\omega) \right) \leq M^{1/\alpha_1 \land 1}, \sup_{k \in \lbrace 0, \ldots, n-1 \rbrace}  |\Delta W_{k}^{i}(\omega)| \leq 1  \right \rbrace,
\end{align*}
the set of events such that the processes $D^{i}_{n, \lambda}: \Omega \to [0,\infty)$, given by 
\begin{align*}
D^{i}_{n, \lambda} := \left(\lambda^2 +   |Y_{0}^{i,N}|^2 \right) \exp \left[ \lambda+  \sup_{u \in \lbrace 0, \ldots, n \rbrace}  \sum_{k=u}^{n-1} \lambda\left(| \Delta W_k^{i}|^2 + \tilde{A}_k^{i} +  \alpha^{i}_k \right)  \right],
\end{align*}
and $D_{n, \lambda}: \Omega \to [0,\infty)$ (precisely defined at the end of the proof), for $i \in \mathbb{S}_N$, $n \in \lbrace 0, \ldots, M \rbrace$, are small enough. Further, $\lambda \geq 1$ is some sufficiently large constant and $\tilde{A}_n^{i}: \Omega \to [0,\infty)$, for $i \in \mathbb{S}_N$, $n \in \lbrace 0, \ldots, M \rbrace$, will be defined at a later stage of the proof.

Also, we introduce the quantity
\begin{align*}
& \alpha^{i}_n := {\bf I}_ {\lbrace |Y^{i,N}_{t_{n}}| \geq c  \rbrace} \left \langle  \frac{ Y_{t_n}^{i,N}}{|Y_{t_n}^{i,N}|}, \frac{  \si(Y_{t_n}^{i,N},\mu_{t_n}^{Y^{\cdot,N}})\Delta W_n^{i} }{|Y_{t_n}^{i,N}|}  \right \rangle,
\end{align*}
where $c \geq 1$ and its role will become clearer at a later stage of the proof. Let $Y^{i,N}_{t_n}$ be given by (\ref{eq:MilsteinDiscrete}), then we aim at showing that
\begin{equation}\label{eq:eqDominator}
{\bf I}_{\Omega^{i}_{n, \lambda} \cap \Omega^{i}_n} |Y^{i,N}_{t_n}|^2 \leq D^{i}_{n, \lambda} + D_{n, \lambda},
\end{equation} 
for all $i \in \mathbb{S}_N$, $n \in \lbrace 0, \ldots, M \rbrace$ and $M \in \mathbb{N}$. 
Inequality (\ref{eq:eqDominator}) is a pathwise estimate of each particle on the set of events
$\Omega^{i}_{n, \lambda} \cap \Omega^{i}_n$ and is a version of \cite[Lemma 3.1]{HJK} adapted to particle systems.
This lemma is the crucial result to obtain $p$-th moment bounds of tamed schemes.

Note that on $\Omega^{i}_{n+1} \cap \Omega^{i}_{n+1, \lambda} \cap \lbrace \omega | \  |Y^{i,N}_{t_n}(\omega)| \leq c \rbrace $ we have due to ({\bf A}$_b^1$), ({\bf A}$_b^4$), ({\bf
A}$_\si^1$), ({\bf A}$_\si^2$) and ({\bf A}$_\si^4$) that there exists some constant $C >0$ such that
\begin{align*}
| Y^{i,N}_{t_{n+1}}| & \leq  |Y^{i,N}_{t_{n}}| + \delta |b(Y_{t_n}^{i,N},\mu_{t_n}^{Y^{\cdot,N}})| + C \| \sigma(Y_{t_n}^{i,N},\mu_{t_n}^{Y^{\cdot,N}}) \| + C \| (\nabla \sigma) \sigma (Y_{t_n}^{i,N},\mu_{t_n}^{Y^{\cdot,N}}) \| \\
& \quad + C \frac{1}{N} \sum_{j=1}^{N} \| D^{L} \sigma(Y_{t_n}^{i,N},\mu_{t_n}^{Y^{\cdot,N}})(Y_{t_n}^{j,N}) \sigma(Y_{t_n}^{j,N},\mu_{t_n}^{Y^{\cdot,N}}) \|   \\
& \leq c + C \delta   (1 + |Y^{i,N}_{t_{n}}|^{\alpha_1})|Y^{i,N}_{t_{n}}| + C|Y^{i,N}_{t_{n}}| + C   \\ 
& \leq \lambda,
\end{align*}
for all $n \in \lbrace 0, \ldots, M-1 \rbrace$ and $M \in \mathbb{N}$, where $\lambda \geq 1$ is chosen large enough and depends on the constants appearing in the assumptions for $b$ and $\sigma$ and the terminal time $T>0$.

Further, we obtain from standard inequalities, 
\begin{align*}
| Y^{i,N}_{t_{n+1}}|^2 &=  \left|Y^{i,N}_{t_{n}} +  b_\dd(Y_{t_n}^{i,N},\mu_{t_n}^{Y^{\cdot,N}})\delta + \int_{t_n}^{t_{n+1}} \Upsilon_s^i \d W_s^i \right|^2 \\
& \leq  |Y^{i,N}_{t_{n}}|^2 +  2\delta^2 |b(Y_{t_n}^{i,N},\mu_{t_n}^{Y^{\cdot,N}})|^2  +  2\left|\int_{t_n}^{t_{n+1}} \Upsilon_s^i \d W_s^i \right|^2   \\
& \quad + \frac{2\delta}{1 + \delta|b(Y_{t_n}^{i,N},\mu_{t_n}^{Y^{\cdot,N}})|} \left \langle Y^{i,N}_{t_{n}} , b(Y_{t_n}^{i,N},\mu_{t_n}^{Y^{\cdot,N}}) \right \rangle + 2\left \langle Y^{i,N}_{t_{n}}, \int_{t_n}^{t_{n+1}} \Upsilon_s^i \d W_s^i \right \rangle.
\end{align*}
In the sequel, we will need the set of events, for $c$ chosen appropriately,
\begin{equation*}
\tilde{\Omega}^{i}_{n+1} := \lbrace \omega \in \Omega| \  M^{1/2\alpha_1 \land 1/2} \geq |Y_{t_n}^{i,N}(\omega)|  \geq c \rbrace.
\end{equation*}

Using the global Lipschitz assumption ({\bf A}$_\si^1$) for $\sigma$ and the growth condition on $(\nabla \sigma)\sigma$, see ({\bf A}$_\si^4$), allows us to deduce that  
\begin{align}\label{eq:BoundsSigma}
 \|\sigma(Y_{t_n}^{i,N}, \mu_{t_n}^{Y^{\cdot,N}})\|^2 \1  |Y_{t_n}^{i,N}|^2, \qquad \qquad 
\| \nn \sigma(Y_{t_n}^{i,N},\mu_{t_n}^{Y^{\cdot,N}}) \sigma(Y_{t_n}^{i,N}, \mu_{t_n}^{Y^{\cdot,N}}) \|^2 \1  |Y_{t_n}^{i,N}|^2. \
\end{align}
From the one-sided Lipschitz assumption and the polynomial growth (in the state variable) of the drift term in ({\bf A}$_b^1$),
\begin{align*}
\left \langle Y^{i,N}_{t_{n}} , b(Y_{t_n}^{i,N},\mu_{t_n}^{Y^{\cdot,N}}) \right \rangle  \1  |Y^{i,N}_{t_{n}}|^2, \qquad 
|b(Y_{t_n}^{i,N},\mu_{t_n}^{Y^{\cdot,N}})|^2 \1  M  |Y_{t_n}^{i,N}|^{2}.
\end{align*}
By virtue of (\ref{eq:BoundsSigma}), we obtain the estimate 
\begin{align*}
\left|\int_{t_n}^{t_{n+1}} \Upsilon_s^{i} \d W_s^{i} \right|^2 \1 |Y_{t_n}^{i,N}|^2 |\Delta W^{i}_n|^2  + |Y_{t_n}^{i,N}|^2 A_n^{i} + \left(A_n^{i,N} \right)^2  \quad \text{ on }  \quad\tilde{\Omega}^{i}_{n+1} \cap \Omega^{i}_{n+1}.
\end{align*}
Note that on $\tilde{\Omega}^{i}_{n+1}$
\begin{align*}
 \left \langle Y^{i,N}_{t_{n}}, \int_{t_n}^{t_{n+1}} \Upsilon_s^{i} \d W_s^{i} \right \rangle  & \1  \left \langle Y^{i,N}_{t_{n}},  \si(Y_{t_n}^{i,N},\mu_{t_n}^{Y^{\cdot,N}})\Delta W_n^{i} \right \rangle + |Y_{t_n}^{i,N}|^2 A_n^{i} + |Y_{t_n}^{i,N}| A_n^{i,N}. 
\end{align*}
Further, for $r \neq q$ we have the following decomposition of the L\'{e}vy area
\begin{align*}
\int_{t_n}^{t_{n+1}} \Delta W_s^{i,r} \mathrm{d}W_s^{i,q} = \Delta W_n^{i,r} \Delta H_n^{q} - \Delta W_n^{i,q}\Delta H_n^{r} + L_n^{r,q},
\end{align*}
where $L_n^{r,q}$ is a logistic random variable with zero mean and variance $\frac{1}{12}\delta^2$ and $\Delta H_n^{q}$ (the space time L\'{e}vy area) is normally distributed with zero mean and variance $\frac{1}{12}\delta^2$ (see \cite{FOL,TERRY} and references therein for details). 
Using this, we may write
\begin{align*}
|Y_{t_n}^{i,N}|^2 A_n^{i} &\leq |Y_{t_n}^{i,N}|^2 \left( \sum_{r,q=1}^m \left| \Delta W_n^{i,r} \Delta H_n^{q} \right| + \sum_{r,q=1}^m \left| \Delta W_n^{i,q}\Delta H_n^{r} \right| + \sum_{r,q=1}^m \left| L_n^{r,q} \right|  \right) \\
& \1 |Y_{t_n}^{i,N}|^2 \left( \sum_{r=1}^m \left| (\Delta W_n^{i,r})^2 \right|  +  \sum_{q=1}^m \left|(\Delta H_n^{q})^2 \right| + \sum_{q=1}^m  \left|(\Delta W_n^{i,q})^2 \right|  +  \sum_{r=1}^m \left| (\Delta H_n^{r})^2 \right| + \sum_{r,q=1}^m \left| L_n^{r,q} \right|  \right) \\
& =|Y_{t_n}^{i,N}|^2  \tilde{A}_n^{i}.
\end{align*}
Therefore, on $\tilde{\Omega}^{i}_{n+1} \cap \Omega^{i}_{n+1}$ we have, for some constant $C>0$,
\begin{align*}
| Y^{i,N}_{t_{n+1}}|^2 &\leq  |Y^{i,N}_{t_{n}}|^2 + C\delta(T  + 1)  |Y^{i,N}_{t_{n}}|^2  \\
& \quad + C |Y_{t_n}^{i,N}|^2( \delta +  |\Delta W^{i}_n|^2) + 2  \left \langle Y^{i,N}_{t_{n}}, \si(Y_{t_n}^{i,N},\mu_{t_n}^{Y^{\cdot,N}})\Delta W_n^{i} \right \rangle  + C|Y_{t_n}^{i,N}|^2 \tilde{A}_n^{i} + \frac{C}{c}  |Y_{t_n}^{i,N}|^2 A_n^{i,N},
\end{align*}
where $c$ will be set to $C$. Hence, it can be shown that there is a sufficiently large $\lambda \geq 1$ such that
\begin{align}\label{eq:Iteration}
| Y^{i,N}_{t_{n+1}}|^2  \leq & |Y^{i,N}_{t_{n}}|^2 \left( \exp \left( \frac{\lambda}{M} + \lambda  (| \Delta W^{i}_n|^2  + \tilde{A}_n^{i} +  \alpha^{i}_n)  \right) + A_n^{i,N} \right)
\end{align}
on $\tilde{\Omega}^{i}_{n+1} \cap \Omega^{i}_{n+1}$. For a sufficiently large $M$ and some positive real parameter $a$ we have
\begin{align}\label{eq:harmonic}
\mathbb{E}\left( \exp(a|L_n^{p,q}|) \right)  = \frac{2 \pi + a \delta \left( H \left(-\frac{a \delta}{4 \pi} \right) - H \left(-\frac{a \delta}{4 \pi} - \frac{1}{2} \right) \right) }{2 \pi} \leq 1 + C \delta \leq \exp(C\delta), 
\end{align}
where $C>0$ is independent of $\delta$ and $H(\cdot)$ is the harmonic number for real values. 

For the additional summand $A_n^{i,N}$ in (\ref{eq:Iteration}), which does not appear in \cite[Lemma 3.1]{HJK}, we remark that defining
\begin{align*}
& I^{i}_n := \exp \left( \frac{\lambda}{M} +  \lambda  (| \Delta W^{i}_n|^2  + \tilde{A}_n^{i} +  \alpha^{i}_n) \right),  
\end{align*}
and iterating (\ref{eq:Iteration}) gives, for an integer $q < n$,
\begin{align*}
| Y^{i,N}_{t_{n+1}}|^2 & \leq  |Y^{i,N}_{t_{n}}|^2 I^{i}_n + |Y^{i,N}_{t_{n}}|^2 A_n^{i,N}  
\leq \ldots \leq  |Y^{i,N}_{t_{q}}|^2 \prod_{l=q}^{n}(I^{i}_l + A_l^{i,N} ).   
\end{align*}
It remains to analyse
\begin{align*}
\prod_{l=q}^{n}(I^{i}_l + A_l^{i,N}) = \prod_{l=q}^{n}I^{i}_l + \prod_{l=q}^{n} A_l^{i,N} + \text{ mixed terms},
\end{align*}   
where the first product can be estimated by $D^{i}_{n+1, \lambda}$, which can be treated as in \cite[Lemma 3.5]{HJK}. The second product and the mixed terms, in total at most $2^{M}$ terms, have the form 
\begin{equation}\label{eq:mixedterms}
A_n^{i,N} \ldots I^{i}_{l_1} \ldots A_{l_2}^{i,N} I^{i}_{l_2-1} \ldots  A_{l_3}^{i,N} \ldots  A_{q}^{i,N}. 
\end{equation}
Let now $\lbrace l_1, \tilde{l}_2, \ldots, \tilde{l}_u \rbrace =:L$, for $u \in \lbrace 1, \ldots, n \rbrace$, denote the set of indices indicating an appearance of $I^{i}$ in (\ref{eq:mixedterms}). We can estimate
\begin{align}\label{eq:mixedterms2}
& A_n^{i,N} \ldots I^{i}_{l_1} \ldots A_{l_2}^{i,N} I^{i}_{l_2-1} \ldots  A_{l_3}^{i,N} \ldots  A_{q}^{i,N}  \notag \\
& \leq \left(\lambda^2 +   |Y_{0}^{i,N}|^2 \right) \exp \left[ \lambda+  \sup_{\tilde{L} \subseteq L} \sum_{k \in \tilde{L}}\lambda\left(| \Delta W_k^{i}|^2 + \tilde{A}_k^{i} +  \alpha^{i}_k \right)  \right] A_n^{i,N} \ldots A_{l_2}^{i,N}  \ldots  A_{l_3}^{i,N} \ldots  A_{q}^{i,N}.
\end{align}
This allows us to introduce $D_{n+1, \lambda}:=\prod_{l=0}^{n} A_l^{i,N} + \text{ mixed terms}$, where the mixed terms are defined by (\ref{eq:mixedterms2}). Further, note that all factors of the form (\ref{eq:mixedterms2}) are in expectation of order $\mathcal{O}(\delta^{r})$, where $r$ denotes the number of appearances of a factor of the form $A_l^{i,N}$, due to H\"{o}lder's inequality and $\mathbb{E} \left(A_l^{i,N} \right)^{q} \1 \delta^q$ for $q \geq 1$. Hence, in total, we obtain a term of order 
\begin{equation*}
\sum_{k=0}^{M} \binom{M}{k} \delta^{M-k},
\end{equation*}
which, as $M \to \infty$, tends to e$^{T}$ (Euler's constant).

An inductive argument as in \cite[Lemma 3.1]{HJK} allows us to deduce (\ref{eq:eqDominator}).
Furthermore, due to (\ref{eq:harmonic}) and (\ref{eq:mixedterms2}), analogous statements to \cite[Lemmas 3.2--3.8]{HJK} still hold true
\end{proof}

\subsection{Some auxiliary lemmata}
\begin{lem}\label{Lem1}
{\rm Assume ({\bf A}$_b^1$) and ({\bf A}$_\si^1$)--({\bf A}$_\si^3$).
Then, for all $p \ge 1$, there is a constant $C>0$ such that
\begin{equation}\label{F1}
\begin{split}
\Lambda_t^{i,p}:&=\int_0^t(\E\|\si(X_s^{i,N},\mu_s^{X^{\cdot,N}
})-\Upsilon_s^{i}\|^{2p})^{\ff{1}{p}} \d s  \le
C\Big\{\dd^2+\int_0^t (\E|Z^{i,N}_s|^{2p})^{\ff{1}{p}}\d
s\Big\}, \qquad t \geq 0.
\end{split}
\end{equation}
 }
\end{lem}
\begin{proof}
From ({\bf A}$_\si^1$) and Minkowski's inequality,
we derive that
\begin{equation*}
\begin{split}
\Lambda_t^{i,p} &\1\int_0^t (\E\|\si(X_s^{i,N},\mu_s^{X^{\cdot,N}
})-\si(Y_s^{i,N},\mu_s^{Y^{\cdot,N}}) \|^{2p})^{\ff{1}{p}}\d
s+\int_0^t (\E\|\Gamma_s^i\|^{2p})^{\ff{1}{p}}\d s\\
&\1\int_0^t  (\E|Z^{i,N}_s|^{2p})^{\ff{1}{p}}\d
s+\ff{1}{N}\sum_{j=1}^N\int_0^t(\E|Z^{j,N}_s|^{2p})^{\ff{1}{p}}\d s+\int_0^t(\E\|\Gamma_s^i\|^{2p})^{\ff{1}{p}}\d s\\
&\1\int_0^t (\E|Z^{i,N}_s|^{2p})^{\ff{1}{p}}\d s
+\int_0^t(\E\|\Gamma_s^i\|^{2p})^{\ff{1}{p}}\d s,
\end{split}
\end{equation*}
where in the last display we used the fact that
$Z_t^{i,N}, i \in\mathbb{S}_N, t \in [0,T]$, are identically distributed.
Consequently, to derive \eqref{F1}, it is sufficient to show that
\begin{equation}\label{F22}
(\E\|\Gamma_t^i\|^{2p})^{\ff{1}{p}} \1 \dd^2.
\end{equation}
In the sequel, we aim at verifying \eqref{F22}. According
to \cite[Proposition 3.1]{CCD}, we have
\begin{equation}\label{W5}
\begin{split}
\frac{ \partial \si_{kl}}{\partial x_j} (x_i,\hat\mu^{{\bf x}, N})&= \frac{ \partial \si_{kl}(\cdot, \hat\mu^{{\bf x}, N})}{\partial x_j}(x_i)\dd_{j,i}+\ff{1}{N}D^L \si_{kl}(x_i,\hat\mu^{{\bf x}, N})(x_j),\\
\frac{ \partial^2 \si_{kl}}{\partial x_j^2}(x_i, \hat \mu^{{\bf x}, N})
&=\Big\{ \frac{\partial^2 \si_{kl}(\cdot, \hat\mu^{{\bf x}, N})}{\partial x_j^2}(x_i)+\ff{2}{N}  \frac{ \partial
\{D^L\si_{kl}(\cdot,\hat\mu^{{\bf x},N})(x_i)\}}{\partial x_j}(x_i)\Big\}\dd_{j,i}\\
&\quad+\ff{1}{N^2}(D^L)^2\si_{kl}(x_i,\hat\mu^{{\bf
x}, N})(x_j,x_j)+\ff{1}{N}\nn \{D^L\si_{kl}(x_i,\hat\mu^{{\bf
x}, N})(\cdot)\}(x_j),~i, j\in \mathbb{S}_N,
\end{split}
\end{equation}
where $\sigma_{kl}$ is a component of $\sigma$, and $k \in \lbrace 1, \ldots, m  \rbrace$, $l \in \lbrace 1, \ldots, d \rbrace$. Also, $\dd_{j,i}=1$, for $j=i$ and $\dd_{j,i}=0, $ for $j \neq i$. It\^o's formula in combination with  \eqref{W5} yields
\begin{align*}
\Gamma_t^i &=\int_{t_\dd}^t \nn\si(\cdot,\mu_s^{Y^{\cdot,N}
})(Y_s^{i,N}) b_\dd(Y_{s_\dd}^{i,N},\mu_{s_\dd}^{Y^{\cdot,N}})\d
s\\
&\quad+\ff{1}{N}\sum_{j=1}^N\int_{t_\dd}^tD^L\si(Y_s^{i,N},\mu_s^{Y^{\cdot,N}
})(Y_s^{j,N})b_\dd(Y_{s_\dd}^{j,N},\mu_{s_\dd}^{Y^{\cdot,N}})\d s\\
&\quad+\ff{1}{2}\sum_{l=1}^m\int_{t_\dd}^t\Big\{\nn^2\si(\cdot,\mu_s^{Y^{\cdot,N}
})(Y_s^{i,N})+\ff{2}{N}\nn\{ D^L\si(\cdot,\mu_s^{Y^{\cdot,N}
})(Y_s^{i,N})\}(Y_s^{i,N})\\
&\quad+\ff{1}{2}\sum_{k=1}^N
\Big(\ff{1}{N^2}(D^L)^2\si(Y_s^{i,N},\mu_s^{Y^{\cdot,N}
})(Y_s^{k,N},Y_s^{k,N})\\
&\quad+\ff{1}{N}\nn\{ D^L\si(Y_s^{i,N},\mu_s^{Y^{\cdot,N}
})(\cdot)\}(Y_s^{k,N})\Big)\Big\}
(\Upsilon_s^ie_l,\Upsilon_s^ie_l)\d s\\
&\quad+\int_{t_\dd}^t \{\nn\si(\cdot,\mu_s^{Y^{\cdot,N}
})(Y_s^{i,N})-\nn\si(\cdot,\mu_{s_\dd}^{Y^{\cdot,N}
})(Y_{s_\dd}^{i,N})\} \si(Y_{s_\dd}^{i,N},\mu_{s_\dd}^{Y^{\cdot,N}}) \d W_s^i\\
&\quad+\int_{t_\dd}^t \nn\si(\cdot,\mu_s^{Y^{\cdot,N}
})(Y_s^{i,N})\int_{s_\dd}^s(\nn\si(\cdot,\mu_{r_\dd}^{Y^{\cdot,N}})(Y_{r_\dd}^{i,N})\si(Y_{r_\dd}^{i,N},\mu_{r_\dd}^{Y^{\cdot,N}})\d
W_r^i\d W_s^i\\
&\quad+\ff{1}{N}\sum_{j=1}^N\int_{t_\dd}^tD^L\si(Y_s^{i,N},\mu_s^{Y^{\cdot,N}
})(Y_s^{j,N})\int_{s_\dd}^s\nn\si(\cdot,\mu_{r_\dd}^{Y^{\cdot,N}})(Y_{r_\dd}^{j,N})\si(Y_{r_\dd}^{j,N},\mu_{r_\dd}^{Y^{\cdot,N}})\d
W_r^j\d W_s^j\\
&\quad+\ff{1}{N}\sum_{j=1}^N\int_{t_\dd}^t\nn\si(\cdot,\mu_s^{Y^{\cdot,N}
})(Y_s^{i,N})\int_{s_\dd}^sD^L\si(Y_{r_\dd}^{i,N},\mu_{r_\dd}^{Y^{\cdot,N}})(Y_{r_\dd}^{j,N})\si(Y_{r_\dd}^{j,N},
\mu_{r_\dd}^{Y^{\cdot,N}})\d W_r^j \d W_s^i\\
&\quad+\ff{1}{N^2}\sum_{j=1}^N\int_{t_\dd}^tD^L\si(Y_s^{i,N},\mu_s^{Y^{\cdot,N}
})(Y_s^{j,N})\int_{s_\dd}^sD^L\si(Y_{r_\dd}^{i,N},\mu_{r_\dd}^{Y^{\cdot,N}})(Y_{r_\dd}^{j,N})\si(Y_{r_\dd}^{j,N},
\mu_{r_\dd}^{Y^{\cdot,N}})\d W_r^j \d W_s^i\\
&\quad+\ff{1}{N}\sum_{j=1}^N\int_{t_\dd}^t\{D^L\si(Y_s^{i,N},\mu_s^{Y^{\cdot,N}})(Y_s^{j,N})-D^L
\si(Y_{s_\dd}^{i,N},\mu_{s_\dd}^{Y^{\cdot,N}})(Y_{s_\dd}^{j,N})\}\si(Y_{s_\dd}^{j,N},
\mu_{s_\dd}^{Y^{\cdot,N}})\d W_s^j \\
&=:\sum_{k=1}^9\Theta_t^{k,i},
\end{align*}
where $(e_l)_{1\le l\le m}$ is the standard orthogonal basis of
$\R^m.$ Note that Lemma \ref{TH:TH1} implies that there is a constant $C>0$ such that
\begin{equation}\label{E1}
\sup_{i \in \mathbb{S}_N} \sup_{0 \leq t \leq T} \E |Y^{i,N}_t|^p \leq C.
\end{equation}
Due to (2) and (3) of ({\bf A}$_b^1$) and ({\bf A}$_\si^1$), one has
\begin{equation}\label{W7}
\begin{split}
 |b(x,\mu)|+\|\si(x,\mu)\|\1
1+|x|^{1+\aa_1}+\mu(|\cdot|^2)^{\ff{1}{2}}, \qquad x, y \in \R^d,\mu \in \mathscr
{P}_2(\R^d).
\end{split}
\end{equation}
Further, note that ({\bf A}$_\si^1$) implies
\begin{equation}\label{W6}
\|\nn_y\si(x,\mu)\| \1 |y|, \qquad x, y \in \R^d, \mu\in\mathscr
{P}_2(\R^d).
\end{equation}
Obviously, \eqref{E1} and \eqref{W7} yield
\begin{equation}\label{E6}
\sup_{0\le t\le
T}\{\E|b(Y_{t}^{i,N},\mu_{t}^{Y^{\cdot,N}})|^p+\E\|\si(Y_{t}^{i,N},\mu_{t}^{Y^{\cdot,N}})\|^p \} \leq C.
\end{equation}
By H\"older's inequality and Minkowski's inequality, it follows from
\eqref{E1}, \eqref{W6}, \eqref{E6} and
  ({\bf A}$_\si^2$)  that
\begin{align}\label{R2}
&\sum_{k=1}^3(\E\|\Theta_t^{k,i}\|^{2p})^{\ff{1}{p}} \nonumber \\
&\1\dd\int_{t_\dd}^t
(\E\|\nn\si(\cdot,\mu_s^{Y^{\cdot,N}
})(Y_s^{i,N})\|^{4p})^{\ff{1}{2p}}
\d s (\E |b(Y_{t_\dd}^{i,N},\mu_{t_\dd}^{Y^{\cdot,N}}) |^{4p})^{\ff{1}{2p}} \nonumber \\
&\quad+\ff{\dd}{N}\sum_{j=1}^N\int_{t_\dd}^t
(\E\|D^L\si(Y_s^{i,N},\mu_s^{Y^{\cdot,N}
})(Y_s^{j,N})\|^{4p})^{\ff{1}{2p}}\d s (\E |b(Y_{t_\dd}^{j,N},\mu_{t_\dd}^{Y^{\cdot,N}}) |^{4p})^{\ff{1}{2p}} \nonumber \\
&\quad+\dd\int_{t_\dd}^t\Big\{(\E\|\nn^2\si(\cdot,\mu_s^{Y^{\cdot,N}
})(Y_s^{i,N})\|^{4p})^{\ff{1}{2p}}+(\E\|\nn\{
D^L\si(\cdot,\mu_s^{Y^{\cdot,N}
})(Y_s^{i,N})\}(Y_s^{i,N})\|^{4p})^{\ff{1}{2p}} \nonumber \\
&\quad+ \ff{1}{N^2}\sum_{k=1}^N
(\E\|(D^L)^2\si(Y_s^{i,N},\mu_s^{Y^{\cdot,N}
})(Y_s^{k,N},Y_s^{k,N})\|^{4p})^{\ff{1}{2p}} \nonumber \\
&\quad+\ff{1}{N}\sum_{k=1}^N(\E\|\nn\{
D^L\si(Y_s^{i,N},\mu_s^{Y^{\cdot,N}
})(\cdot)\}(Y_s^{k,N})\|^{4p})^{\ff{1}{2p}} \Big\}
(\E\|\Upsilon_s^i\|^{4p})^{\ff{1}{2p}}\d s \nonumber \\
&\1\dd^2.
\end{align}
Taking \eqref{E1}-\eqref{E6}, and ({\bf A}$_\si^4$) into
consideration yields
\begin{equation*}
\sup_{0\le t\le T} \E |Y_t^{i,N}-Y_{t_\dd}^{i,N}|^q \1
\dd^{\ff{q}{2}},~~~q \geq 1,~~i\in\mathbb{S}_N.
\end{equation*}
Therefore, employing BDG's inequality in combination with H\"older's inequality, gives that
\allowdisplaybreaks
\begin{align}\label{P1}
&\sum_{k=4}^9(\E\|\Theta_t^{k,i}\|^{2p})^{\ff{1}{p}} \nonumber\\
&\1\dd^2(\E\|\nn\si(\cdot,\mu_{t_\dd}^{Y^{\cdot,N}})(Y_{t_\dd}^{i,N}) \si(Y_{t_\dd}^{i,N},\mu_{t_\dd}^{Y^{\cdot,N}}) \|^{4p})^{\ff{1}{2p}} \nonumber\\
&\quad+
(\E\|\si(Y_{t_\dd}^{i,N},\mu_{t_\dd}^{Y^{\cdot,N}})\|^{4p})^{\ff{1}{2p}}\int_{t_\dd}^t
(\E\|\nn\si(\cdot,\mu_s^{Y^{\cdot,N}
})(Y_s^{i,N})-\nn\si(\cdot,\mu_{t_\dd}^{Y^{\cdot,N}
})(Y_{t_\dd}^{i,N})\|^{4p})^{\ff{1}{2p}} \d s
\nonumber \\
&\quad+\ff{\dd^2}{N}\sum_{j=1}^N 
(\E\| \nn\si(\cdot,\mu_{t_\dd}^{Y^{\cdot,N}})(Y_{t_\dd}^{i,N}) \si(Y_{t_\dd}^{i,N},\mu_{t_\dd}^{Y^{\cdot,N}}) \|^{4p})^{\ff{1}{2p}} \nonumber \\ 
&\quad+\ff{\dd^2}{N} \sum_{j=1}^N 
(\E\|D^L\si(Y_{t_\dd}^{i,N},\mu_{t_\dd}^{Y^{\cdot,N}})(Y_{t_\dd}^{j,N}) \si(Y_{t_\dd}^{i,N},\mu_{t_\dd}^{Y^{\cdot,N}})\|^{4p})^{\ff{1}{2p}} 
 \nonumber \\
 &\quad+ \ff{\dd^2}{N^2} \sum_{j=1}^N  (\E\|D^L\si(Y_{t_\dd}^{i,N},\mu_{t_\dd}^{Y^{\cdot,N}})(Y_{t_\dd}^{j,N}) \si(Y_{t_\dd}^{j,N},\mu_{t_\dd}^{Y^{\cdot,N}}) \|^{4p})^{\ff{1}{2p}} \nonumber \\
&\quad+\ff{1}{N} \sum_{j=1}^N \int_{t_\dd}^t
(\E\|D^L\si(Y_s^{i,N},\mu_s^{Y^{\cdot,N}})(Y_s^{j,N})-D^L
\si(Y_{t_\dd}^{i,N},\mu_{t_\dd}^{Y^{\cdot,N}})(Y_{t_\dd}^{j,N})\|^{4p})^{\ff{1}{2p}}\d
s \times \nonumber \\
&\quad ~~~~~~~~~~~~~~~~\times(\E\|\si(Y_{t_\dd}^{j,N}, \mu_{t_\dd}^{Y^{\cdot,N}})\|^{4p})^{\ff{1}{2p}} \nonumber \\
& \1\dd^2,
\end{align}
where we exploited ({\bf A}$_\si^2$), ({\bf A}$_\si^3$), and
\eqref{E1}-\eqref{W6} in the last display. Consequently, \eqref{F22}
follows from \eqref{R2} and \eqref{P1}.
\end{proof}
\begin{lem}\label{Lem3}
{\rm Assume ({\bf A}$_b^1$) and ({\bf A}$_\si^1$)--({\bf A}$_\si^2$). Then, for all $\vv>0$, $p \ge 1$, there
exists a constant $C>0$ such that
\begin{equation}\label{p0}
\begin{split}
\Big(\E\Big|\int_0^t\<Z^{i,N}_s,\hat\Upsilon_s^i\>\d
s\Big|^p\Big)^{\ff{1}{p}}\le
\vv\,(\E\|Z^{i,N}\|^{2p}_{\8,t})^{\ff{1}{p}}+C \Big\{\int_0^t
(\E|Z^{i,N}_s|^{2p})^{\ff{1}{p}}\d s+\dd^2\Big\}, \quad t \geq 0,
\end{split}
\end{equation}
where
\begin{equation*}
\hat\Upsilon_t^i:= \nn b(\cdot,\mu_{t_\dd}^{Y^{\cdot,N}
})(Y_{t_\dd}^{i,N}) \si(Y_{t_\dd}^{i,N},\mu_{t_\dd}^{Y^{\cdot,N}}) (
W_t^i-W_{t_\dd}^i).
\end{equation*}
The same result can be shown if $\hat\Upsilon_t^i$ is replaced by
\begin{equation*}
\ff{1}{N}\sum_{j=1}^N D^L
b(Y_{t_\dd}^{i,N},\mu_{t_\dd}^{Y^{\cdot,N}})(Y_{t_\dd}^{j,N})\si(Y_{t_\dd}^{j,N},
\mu_{t_\dd}^{Y^{\cdot,N}})(W_t^j-W_{t_\dd}^j).
\end{equation*}
}\end{lem}

\begin{proof}
Since
\begin{equation*}
\begin{split}
Z_t^{i,N}&=Z_{t_\dd}^{i,N}+\{b(X_{t_\dd}^{i,N},\mu_{t_\dd}^{X^{\cdot,N} })-b
(Y_{t_\dd}^{i,N},\mu_{t_\dd}^{Y^{\cdot,N}})\}(t-t_\dd) +\{b(Y_{t_\dd}^{i,N},\mu_{t_\dd}^{Y^{\cdot,N}
})-b_\dd(Y_{t_\dd}^{i,N},\mu_{t_\dd}^{Y^{\cdot,N}})\}(t-t_\dd) \\
&\quad+\int_{t_\dd}^t\{b(X_s^{i,N},\mu_s^{X^{\cdot,N}
})-b(X_{t_\dd}^{i,N},\mu_{t_\dd}^{X^{\cdot,N}})\}\d s\\
&\quad+\int_{t_\dd}^t\{\si(X_s^{i,N},\mu_s^{X^{\cdot,N}
})-\si(Y_s^{i,N},\mu_s^{Y^{\cdot,N}})\}\d W_s^i + \int_{t_\dd}^t \Gamma_s^i\d
W_s^i,
\end{split}
\end{equation*}
where $\GG^i$ was introduced in \eqref{D3}, we deduce from Minkowski's
inequality that
\begin{align*}
 \Big(\E\Big|\int_0^t\<Z^{i,N}_s,\hat\Upsilon_s^i\>\d
s\Big|^p\Big)^{\ff{1}{p}} &\le
\Big(\E\Big|\int_0^t\<Z^{i,N}_{s_\dd},\hat\Upsilon_s^i\>\d
s\Big|^p\Big)^{1/p}+\Big(\E\Big|\int_0^t\Big\<\int_{s_\dd}^s
\Gamma_r^i\d
W_r^i,\hat\Upsilon_s^i\Big\>\d s\Big|^p\Big)^{\ff{1}{p}}\\
 &\quad+\Big(\E\Big|\int_0^t\Big\<\int_{s_\dd}^s\{b(X_r^{i,N},\mu_r^{X^{\cdot,N}
})-b(X_{s_\dd}^{i,N},\mu_{s_\dd}^{X^{\cdot,N}})\}\d
r,\hat\Upsilon_s^i\Big\>\d s\Big|^p\Big)^{\ff{1}{p}}\\
&\quad+\Big(\E\Big|\int_0^t\<\{b(X_{s_\dd}^{i,N},\mu_{s_\dd}^{X^{\cdot,N}
})-b(Y_{s_\dd}^{i,N},\mu_{s_\dd}^{Y^{\cdot,N}})\}(s-s_\dd),\hat\Upsilon_s^i\>\d
s\Big|^p\Big)^{\ff{1}{p}}\\
&\quad+\Big(\E\Big|\int_0^t\<\{b(Y_{s_\dd}^{i,N},\mu_{s_\dd}^{Y^{\cdot,N}})-b_\dd(Y_{s_\dd}^{i,N},\mu_{s_\dd}^{Y^{\cdot,N}})\}(s-s_\dd),\hat\Upsilon_s^i\>\d
s\Big|^p\Big)^{\ff{1}{p}}\\
&\quad+\Big(\E\Big|\int_0^t\Big\<\int_{s_\dd}^s\{\si(X_r^{i,N},\mu_r^{X^{\cdot,N}
})-\si(Y_r^{i,N},\mu_r^{Y^{\cdot,N}})\}\d
W_r^i,\hat\Upsilon_s^i\Big\>\d s\Big|^p\Big)^{\ff{1}{p}}\\
& =:\sum_{k=1}^6\hat\Lambda_t^{k,i}.
\end{align*}
In what follows, we will estimate
$\hat\Lambda_t^{k,i},k=1,\ldots,6,$ one-by-one. By It\^o's formula,
one has
\begin{equation*}
\d((t-t_{k+1})(W_t^i-W_{t_k}^{i}))=(W_t^i-W_{t_k}^i)\d
t+(t-t_{k+1})\d W_t^i,~ t\in[t_k,t_{k+1}],
\end{equation*}
which implies that
\begin{equation*}
\begin{split}
&\int_0^t\<Z^{i,N}_{s_\dd},\hat\Upsilon_s^i\>\d
s\\
&=\sum_{k=0}^{\lfloor t/\dd\rfloor}\int_{t_k}^{t_{k+1} \wedge
t}\<Z^{i,N}_{t_k},\nn b(\cdot,\mu_{t_k}^{Y^{\cdot,N}
})(Y_{t_k}^{i,N}) \si(Y_{t_k}^{i,N},\mu_{t_k}^{Y^{\cdot,N}})
\d((s-t_{k+1})(W_s^i-W_{t_k}^{i}))\>\\
&\quad-\sum_{k=0}^{\lfloor t/\dd\rfloor}\int_{t_k}^{t_{k+1} \wedge
t}(s-t_{k+1})\<Z^{i,N}_{t_k},\nn b(\cdot,\mu_{t_k}^{Y^{\cdot,N}
})(Y_{t_k}^{i,N}) \si(Y_{t_k}^{i,N},\mu_{t_k}^{Y^{\cdot,N}}) \d
W_s^i\>\\
&=(t-(t_\dd+\dd))\<Z^{i,N}_{t_\dd},\nn
b(\cdot,\mu_{t_\dd}^{Y^{\cdot,N} })(Y_{t_\dd}^{i,N})
\si(Y_{t_\dd}^{i,N},\mu_{t_\dd}^{Y^{\cdot,N}}) (W_t^i-W_{t_\dd}^{i})\>\\
&\quad- \int_0^t(s-(s_\dd+\dd))\<Z^{i,N}_{s_\dd},\nn
b(\cdot,\mu_{s_\dd}^{Y^{\cdot,N} })(Y_{s_\dd}^{i,N})
\si(Y_{s_\dd}^{i,N},\mu_{s_\dd}^{Y^{\cdot,N}}) \d W_s^i\>.
\end{split}
\end{equation*}
Hence, we deduce from Young's inequality and BDG's inequality that
\begin{equation}\label{p4}
\begin{split}
\hat\Lambda_t^{1,i}
 &\le\ff{2\,\vv}{5}\,(\E\|Z^{i,N}\|^{2p}_{\8,t})^{\ff{1}{p}}+\ff{5\dd^2}{4\,\vv}(\E\|\nn
b(\cdot,\mu_{t_\dd}^{Y^{\cdot,N} })(Y_{t_\dd}^{i,N})
\si(Y_{t_\dd}^{i,N},\mu_{t_\dd}^{Y^{\cdot,N}})\|^{4p})^{\ff{1}{2p}}\\
&\quad+\ff{5\dd^2}{4\,\vv}  \int_0^t (\E\|\nn
b(\cdot,\mu_{s_\dd}^{Y^{\cdot,N} })(Y_{s_\dd}^{i,N})
\si(Y_{s_\dd}^{i,N},\mu_{s_\dd}^{Y^{\cdot,N}})\|^{2p})^{\ff{1}{p}}\d
 s\\
 &\le
 \ff{2\,\vv}{5}\,(\E\|Z^{i,N}\|^{2p}_{\8,t})^{\ff{1}{p}}+C(\vv)\dd^2,
\end{split}
\end{equation}
for some constant $C(\vv)>0$, where in the last display we
utilised \eqref{E1}, \eqref{E6} and
\begin{equation}\label{D5}
\|\nn b(\cdot,\mu)(x)\|\11+|x|^{\aa_1},
\end{equation}
due to (2) of ({\bf A}$_b^1$). Via H\"older's inequality, we find from
  \eqref{E1} and \eqref{W7}, together with \eqref{D5} that
\begin{equation}\label{H1}
(\E|\hat\Upsilon_t^i|^{2p})^{\ff{1}{2p}}\1 \dd^{1/2}.
\end{equation}
Next, H\"older's inequality followed by BDG's inequality, gives
\begin{equation}\label{p3}
\begin{split}
\hat\Lambda_t^{2,i}
&\le\int_0^t\Big(\E\Big|\int_{s_\dd}^s \Gamma_r^i\d
W_r^i\Big|^{2p}\Big)^{\ff{1}{2p}}(\E|\hat\Upsilon_s^i|^{2p})^{\ff{1}{2p}}\d
s
 \1 \dd^{1/2} \int_0^t\Big( \delta^{p-1}
\int_{s_{\delta}}^{s} \E\|\Gamma_r^i\|^{2p} \d r\Big)^{\ff{1}{2p}} \d s
\1\dd^2,
\end{split}
\end{equation}
where we used \eqref{H1} in the second step, and \eqref{F22} in the
last step. Using Minkowski's inequality and H\"older's inequality, we
derive from (2) and (3) of ({\bf A}$_b^1$), \eqref{D4}, and \eqref{H1}
 that
\begin{equation}\label{p1}
\begin{split}
\hat\Lambda_t^{3,i}
&\le\int_0^t \int_{s_\dd}^s (\E|b(X_r^{i,N},\mu_r^{X^{\cdot,N}
})-b(X_{s_\dd}^{i,N},\mu_{s_\dd}^{X^{\cdot,N}})|^{2p})^{\ff{1}{2p}}\d
r(\E|\hat\Upsilon_s^i|^{2p})^{\ff{1}{2p}}\d s\\
&\1 \dd^{1/2} \int_0^t
\int_{s_\dd}^s\Big\{(1+(\E|X_r^{i,N}|^{4p\aa_1})^{\ff{1}{4p}}+(\E|X_{s_\dd}^{i,N}|^{4p\aa_1}|^{\ff{1}{4p}}))(\E|X_r^{i,N}-X_{s_\dd}^{i,N}|^{4p})^{\ff{1}{4p}}\\
&\quad+\ff{1}{N}\sum_{j=1}^N (\E|X_r^{j,N}
-X_{s_\dd}^{j,N}|^{2p})^{\ff{1}{2p}}\Big\}\d r \d s\\
&\1\dd^2,
\end{split}
\end{equation}
where in the last estimate we utilised that
\begin{equation*}
\E\|X^{i,N}\|_{\8,T}^q \leq C,
~~~~~\E|X_t^{i,N}-X_{t_\dd}^{i,N}|^q\1
\dd^{q/2},~~~q \geq 1,
\end{equation*}
and that $X_t^{i,N}$, $i\in\mathbb{S}_N, t \in [0,T],$ are
identically distributed. Applying H\"older's inequality and Young's
inequality, we deduce from \eqref{E1} and \eqref{H1} that
\begin{equation}\label{pp5}
\begin{split}
\hat\Lambda_t^{4,i} 
&\le L_b^1\dd \int_0^t(\E(|Z_{s_\dd}^{i,N}|(1+|X_{s_\dd}^{i,N}|^{\aa_1}+|Y_{s_\dd}^{i,N}|^{\aa_1})|\hat\Upsilon_s^i|)^p)^{\ff{1}{p}}\d s\\
&\quad+\ff{L_b^1\dd}{N}\sum_{j=1}^N \int_0^t(\E(|Z_{s_\dd}^{j,N} | |\hat\Upsilon_s^i|)^p)^{\ff{1}{p}}\d s\\
&\le\ff{2\,\vv}{5}(\E\|Z^{i,N}\|^{2p}_{\8,t})^{\ff{1}{p}}+\ff{5(L_b^1)^2t\dd^2}{4\vv}\int_0^t(\E|\hat\Upsilon_s^i|^{2p})^{\ff{1}{p}}\d
s\\
&\quad+\ff{5(L_b^1)^2t\dd^2}{4\vv}\int_0^t(1+(\E|X_{s_\dd}^{i,N}|^{2p\aa_1})^{\ff{1}{2p}}+
(\E|Y_{s_\dd}^{i,N}|^{2p\aa_1})^{\ff{1}{2p}})(\E|\hat\Upsilon_s^i|^{2p})^{\ff{1}{2p}}\d
s
 \\
&\le\ff{2\,\vv}{5}(\E\|Z^{i,N}\|^{2p}_{\8,t})^{\ff{1}{p}}+
C_2(\vv)\dd^2
\end{split}
\end{equation}
for some constant $C(\vv)>0.$ Taking advantage of \eqref{E6} and
\eqref{H1} gives
\begin{equation}\label{p5}
\begin{split}
\hat\Lambda_t^{5,i}\le& \dd^2
\int_0^t(\E(|b(Y_{s_\dd}^{i,N},\mu_{s_\dd}^{Y^{\cdot,N}})| |\hat\Upsilon_s^i|)^p)^{\ff{1}{p}}\d
s\1\dd^2.
\end{split}
\end{equation}
Finally, applying  BDG's, Minkowski's, and
Young's inequalities, and taking ({\bf A}$_\si^1$), \eqref{D4} and
\eqref{H1} into account, leads to
\begin{equation}\label{p2}
\begin{split}
\hat\Lambda_t^{6,i}
&\le\ff{\vv}{10\dd t}\int_0^t  \int_{s_\dd}^s
 (\E\|\si(X_r^{i,N},\mu_r^{X^{\cdot,N}
})-\si(Y_r^{i,N},\mu_r^{Y^{\cdot,N}})\|^{2p})^{\ff{1}{p}}\d r \d s
+\ff{5t\dd}{2\,\vv}\int_0^t (\E|\hat\Upsilon_s^i|^{2p})^{\ff{1}{p}}\d s\\
&\le\ff{\vv}{5}(\E\|Z^{i,N}\|^{2p}_{\8,t})^{\ff{1}{p}}+C(\vv)\dd^2
\end{split}
\end{equation}
for some constant $C(\vv)>0.$ Consequently, \eqref{p0} follows
from \eqref{p4} and \eqref{p3}-\eqref{p2}.
\end{proof}

\begin{lem}\label{Lem2}
{\rm Assume ({\bf A}$_b^1$)--({\bf A}$_b^3$) and ({\bf
A}$_\si^1$)--({\bf A}$_\si^2$). Then, for all $\vv>0$, $p \ge 1$, there
exists a constant $C>0$ such that
\begin{equation}\label{P4}
\begin{split}
\hat\Gamma_t^{i,p}:&=\int_0^t(\E|\<Z^{i,N}_s,b(Y_s^{i,N},\mu_s^{Y^{\cdot,N}})-b(Y_{s_\dd}^{i,N},\mu_{s_\dd}^{Y^{\cdot,N}})\>|^p)^{\ff{1}{p}}\d
s\\
&\le\vv\,(\E\|Z^{i,N}\|^{2p}_{\8,t})^{\ff{1}{p}}+ C\Big\{\int_0^t
(\E|Z^{i,N}_s|^{2p})^{\ff{1}{p}}\d s+\dd^2\Big\},~~~t \geq 0.
\end{split}
\end{equation}
}
\end{lem}

\begin{proof}
Using \eqref{W5} with $\si$ therein replaced by $b,$ we obtain from
It\^o's formula that
\begin{align*}
&b(Y_t^{i,N},\mu_t^{Y^{\cdot,N}})-b(Y_{t_\dd}^{i,N},\mu_{t_\dd}^{Y^{\cdot,N}})\\
&=\int_{t_\dd}^t\nn b(\cdot,\mu_{s_\dd}^{Y^{\cdot,N}
})(Y_{s_\dd}^{i,N})
\si(Y_{s_\dd}^{i,N},\mu_{s_\dd}^{Y^{\cdot,N}}) \d W_s^i\\
&\quad+\ff{1}{N}\sum_{j=1}^N\int_{t_\dd}^tD^L
b(Y_{s_\dd}^{i,N},\mu_{s_\dd}^{Y^{\cdot,N}})(Y_{s_\dd}^{j,N})\si(Y_{s_\dd}^{j,N},
\mu_{s_\dd}^{Y^{\cdot,N}})\d W_s^j\\
&\quad+\int_{t_\dd}^t \nn b(\cdot,\mu_s^{Y^{\cdot,N} })(Y_s^{i,N})
b_\dd(Y_{s_\dd}^{i,N},\mu_{s_\dd}^{Y^{\cdot,N}})\d
s\\
&\quad+\ff{1}{N}\sum_{j=1}^N\int_{t_\dd}^tD^Lb(Y_s^{i,N},\mu_s^{Y^{\cdot,N}
})(Y_s^{j,N})b_\dd(Y_{s_\dd}^{j,N},\mu_{s_\dd}^{Y^{\cdot,N}})\d s\\
&\quad+\ff{1}{N}\sum_{j=1}^N\int_{t_\dd}^t\{D^Lb(Y_s^{i,N},\mu_s^{Y^{\cdot,N}})(Y_s^{j,N})-D^L
b(Y_{s_\dd}^{i,N},\mu_{s_\dd}^{Y^{\cdot,N}})(Y_{s_\dd}^{j,N})\}\si(Y_{s_\dd}^{j,N},
\mu_{s_\dd}^{Y^{\cdot,N}})\d W_s^j \\
&\quad+\ff{1}{2}\sum_{l=1}^m\int_{t_\dd}^t\bigg\{\nn^2b(\cdot,\mu_s^{Y^{\cdot,N}
})(Y_s^{i,N})+\ff{2}{N}\nn\{ D^Lb(\cdot,\mu_s^{Y^{\cdot,N}
})(Y_s^{i,N})\}(Y_s^{i,N})\\
&\quad+\ff{1}{2}\sum_{k=1}^N
\Big(\ff{1}{N^2}(D^L)^2b(Y_s^{i,N},\mu_s^{Y^{\cdot,N}
})(Y_s^{k,N},Y_s^{k,N})\\
&\quad+\ff{1}{N}\nn\{ D^Lb(Y_s^{i,N},\mu_s^{Y^{\cdot,N}
})(\cdot)\}(Y_s^{k,N})\Big)\bigg\} (\Upsilon_s^ie_l,\Upsilon_s^ie_l)
\d s\\
&\quad+\int_{t_\dd}^t \Big\{\nn b(\cdot,\mu_s^{Y^{\cdot,N}
})(Y_s^{i,N})-\nn b(\cdot,\mu_{s_\dd}^{Y^{\cdot,N}
})(Y_{s_\dd}^{i,N})\Big\} \si(Y_{s_\dd}^{i,N},\mu_{s_\dd}^{Y^{\cdot,N}}) \d W_s^i\\
&\quad+\int_{t_\dd}^t \nn b(\cdot,\mu_s^{Y^{\cdot,N}
})(Y_s^{i,N})\int_{s_\dd}^s(\nn\si(\cdot,\mu_{r_\dd}^{Y^{\cdot,N}})(Y_{r_\dd}^{i,N})\si(Y_{r_\dd}^{i,N},\mu_{r_\dd}^{Y^{\cdot,N}})\d
W_r^i\d W_s^i\\
&\quad+\ff{1}{N}\sum_{j=1}^N\int_{t_\dd}^t\nn
b(\cdot,\mu_s^{Y^{\cdot,N}
})(Y_s^{i,N})\int_{s_\dd}^sD^L\si(Y_{r_\dd}^{i,N},\mu_{r_\dd}^{Y^{\cdot,N}})(Y_{r_\dd}^{j,N})\si(Y_{r_\dd}^{j,N},
\mu_{r_\dd}^{Y^{\cdot,N}})\d W_r^j \d W_s^i\\
&\quad+\ff{1}{N}\sum_{j=1}^N\int_{t_\dd}^tD^Lb(Y_s^{i,N},\mu_s^{Y^{\cdot,N}
})(Y_s^{j,N})\int_{s_\dd}^s\nn
\si(\cdot,\mu_{r_\dd}^{Y^{\cdot,N}})(Y_{r_\dd}^{j,N})\si(Y_{r_\dd}^{j,N},\mu_{r_\dd}^{Y^{\cdot,N}})\d
W_r^j\d W_s^j\\
&\quad+\ff{1}{N^2}\sum_{j=1}^N\int_{t_\dd}^tD^Lb(Y_s^{i,N},\mu_s^{Y^{\cdot,N}
})(Y_s^{j,N})\int_{s_\dd}^sD^L\si(Y_{r_\dd}^{i,N},\mu_{r_\dd}^{Y^{\cdot,N}})(Y_{r_\dd}^{j,N})\si(Y_{r_\dd}^{j,N},
\mu_{r_\dd}^{Y^{\cdot,N}})\d W_r^j \d W_s^i\\
&=:\sum_{l=1}^{11}\Psi_t^{l,i},
\end{align*}
where $\Upsilon^i$ was given in \eqref{D3}. Therefore, we have
\begin{equation*}
\begin{split}
\hat\Gamma_t^{i,p}\le\sum_{l=1}^{11}\int_0^t(\E|\<Z^{i,N}_s,\Psi_s^{l,i}\>|^p)^{1/p}\d
s.
\end{split}
\end{equation*}
By virtue of Lemma \ref{Lem3}, one has
\begin{equation}\label{PP5}
\int_0^t(\E|\<Z^{i,N}_s,\Psi_s^{1,i}\>|^p)^{\ff{1}{p}}\d s\le
\ff{\vv}{8}(\E\|Z^{i,N}\|^{2p}_{\8,t})^{\ff{1}{p}}+
C\Big\{\int_0^t (\E|Z^{i,N}_s|^{2p})^{\ff{1}{p}}\d
s+\dd^2\Big\},
\end{equation}
for some constant $C>0$ and $\epsilon>0$. On the other hand, following the
steps of the proof of Lemma \ref{Lem3}, there exists a constant $C>0$
such that
\begin{equation}\label{P6}
\int_0^t\E|\<Z^{i,N}_s,\Psi_s^{2,i}\>|^p)^{\ff{1}{p}}\d s\le
\ff{\vv}{8}(\E\|Z^{i,N}\|^{2p}_{\8,t})^{\ff{1}{p}}+ 
C\Big\{\int_0^t (\E|Z^{i,N}_s|^{2p})^{\ff{1}{p}}\d s+\dd^2\Big\}.
\end{equation}
Next, by Young's inequality, it follows
that
\begin{equation*}
\sum_{l=3}^{11}\int_0^t(\E(|Z^{i,N}_s|^p |\Psi_s^{l,i}|^p))^{\ff{1}{p}}\d
s\le\ff{\vv}{2}(\E\|Z^{i,N}\|^{2p}_{\8,t})^{\ff{1}{p}}+\ff{1}{2\vv}\sum_{l=3}^{11}\int_0^t
(\E|\Psi_s^{l,i}|^{2p})^{\ff{1}{p}}\d s.
\end{equation*}
Following analogous arguments to derive \eqref{R2}
and \eqref{P1}, we deduce from ({\bf A}$_b^1$)--({\bf A}$_b^3$) and
({\bf A}$_\si^1$)--({\bf A}$_\si^2$) that
\begin{equation}\label{P7}
\sum_{l=3}^{11}\int_0^t (\E(|\Psi_s^{l,i}|^{2p}))^{\ff{1}{p}}\d
s\1\dd^2.
\end{equation}
As a consequence, \eqref{P4} follows directly from \eqref{PP5},
\eqref{P6} and \eqref{P7}.
\end{proof}

\subsection{Proof of Theorem \ref{TH2:TH2}}

With Lemmas \ref{Lem1} and \ref{Lem2} at hand, we are in a position to
complete the proof of Theorem \ref{TH2:TH2}: 
\begin{proof}
In the sequel, we fix $i\in\mathbb{S}_N$ and recall $t\in[0,T]$. We aim at showing that
\begin{equation}\label{E5}
\E\|Z^{i,N}\|^p_{\8,T}\1\dd^{2p},
\end{equation}
for $p \geq 1$. By It\^o's formula, along with
$Y_0^{i,N}=X_0^{i,N}$, it follows that
\begin{equation*}
\begin{split}
|Z^{i,N}_t|^2&= 2\int_0^t\<Z^{i,N}_s,b(X_s^{i,N},\mu_s^{X^{\cdot,N}
})-b_\dd(Y_{s_\dd}^{i,N},\mu_{s_\dd}^{Y^{\cdot,N}})\>\d s
+2\int_0^t\<Z^{i,N}_s, \d M_s^i\>\\
&\quad+\int_0^t\|\si(X_s^{i,N},\mu_s^{X^{\cdot,N}
})-\Upsilon_s^i\|^2\d s,
\end{split}
\end{equation*}
where $M^i_\cdot$ and $\Upsilon^i_\cdot$ were defined in \eqref{D3}.
Obviously, we have
\begin{equation*}
\begin{split}
&\int_0^t\<Z^{i,N}_s,b(X_s^{i,N},\mu_s^{X^{\cdot,N}
})-b_\dd(Y_{s_\dd}^{i,N},\mu_{s_\dd}^{Y^{\cdot,N}})\>\d
s\\&=\int_0^t\<Z^{i,N}_s,b(X_s^{i,N},\mu_s^{X^{\cdot,N}
})-b(Y_s^{i,N},\mu_s^{X^{\cdot,N} })\>\d s
+\int_0^t\<Z^{i,N}_s,b(Y_s^{i,N},\mu_s^{X^{\cdot,N}
})-b(Y_s^{i,N},\mu_s^{Y^{\cdot,N}})\>\d
s\\
&\quad+\int_0^t\<Z^{i,N}_s,b(Y_{s_\dd}^{i,N},\mu_{s_\dd}^{Y^{\cdot,N}})-b_\dd(Y_{s_\dd}^{i,N},\mu_{s_\dd}^{Y^{\cdot,N}})\>\d
s
+\int_0^t\<Z^{i,N}_s,b(Y_s^{i,N},\mu_s^{Y^{\cdot,N}})-b(Y_{s_\dd}^{i,N},\mu_{s_\dd}^{Y^{\cdot,N}})\>\d
s\\
&=:\Pi_t^{1,i}+\Pi_t^{2,i}+\Pi_t^{3,i}+\Pi_t^{4,i}.
\end{split}
\end{equation*}
From the one-sided Lipschitz condition in ({\bf A}$_b^1$), one
obviously has
\begin{equation}\label{F3}
 \|\Pi^{1,i}\|_{\8,t}\le L^b_1\int_0^t |Z^{i,N}_s|^2\d s.
\end{equation}
Due to (3) of ({\bf A}$_b^1$) and Young's inequality, we obtain from \eqref{D4} 
\begin{equation}\label{F2}
\begin{split}
 \|\Pi^{2,i}\|_{\8,t}&\le \ff{1}{2}\int_0^t |Z^{i,N}_s|^2\d
s+\ff{(L^1_b)^2}{2}\int_0^t
\mathbb{W}_2(\mu_s^{X^{\cdot,N}},\mu_s^{Y^{\cdot,N}})^2\d
s\\
&\le \ff{1}{2}\int_0^t |Z^{i,N}_s|^2\d
s+\ff{(L^1_b)^2}{2N}\sum_{j=1}^N\int_0^t |Z^{j,N}_s|^2\d s.
\end{split}
\end{equation}
In terms of (2) of ({\bf A}$_b^1$), it follows immediately from Young's inequality that
\begin{equation}\label{F5}
\begin{split}
 \|\Pi^{3,i}\|_{\8,t}&\1\int_0^t |Z^{i,N}_s|^2\d
s+\dd^2\int_0^t |b(Y_{s_\dd}^{i,N},\mu_{s_\dd}^{Y^{\cdot,N}})|^2\d
s\\
&\1\int_0^t |Z^{i,N}_s|^2\d
s+\dd^2\int_0^t\Big\{1+|Y_{s_\dd}^{i,N}|^{2(1+\aa_1)}+\ff{1}{N}\sum_{j=1}^N
|Y_{s_\dd}^{j,N}|^2\Big\}\d s.
\end{split}
\end{equation}
Thus,  combining \eqref{F3}, \eqref{F2} with \eqref{F5} and taking
advantage of Lemmas \ref{Lem1} and \ref{Lem2} (with $\epsilon$ small enough), besides \eqref{E1}
and \eqref{E6}, we infer from Minkowski's inequality that, for some
constant $C>0$
\allowdisplaybreaks
\begin{align}\label{P3}
(\E\|Z^{i,N}\|^{2p}_{\8,t})^{\ff{1}{p}}&\le C\bigg\{\int_0^t
 (\E|Z^{i,N}_s|^{2p})^{\ff{1}{p}}\d s+\ff{1}{N}\sum_{j=1}^N\int_0^t
(\E|Z^{j,N}_s|^{2p})^{\ff{1}{p}}\d s \nonumber \\
&\quad+\dd^2\int_0^t
\Big\{1+(\E|Y_{s_\dd}^{i,N}|^{2p(1+\aa_1)})^{\ff{1}{p}}+\ff{1}{N}\sum_{j=1}^N
(\E|Y_{s_\dd}^{j,N}|^{2p})^{\ff{1}{p}}\Big\} \d s \nonumber \\
&\quad+\hat\Gamma_t^{i,p} +\Lambda_t^{i,p}+\Big(\E\Big\|\int_0^\cdot\<Z^{i,N}_s, \d M_s^i\>\Big\|_{\8,t}^p\Big)^{\ff{1}{p}}\bigg\} \nonumber\\
&\le\ff{1}{2
}(\E\|Z^{i,N}\|^{2p}_{\8,t})^{\ff{1}{p}}+C\Big\{\dd^2+\int_0^t
(\E|Z^{i,N}_s|^{2p})^{\ff{1}{p}}\d s \Big\},
\end{align}
where we utilised that $Z_t^{i,N}, i \in\mathbb{S}_N, t \in [0,T],$ are identically
distributed and that, for some constant $C>0$,
\begin{equation*}
\begin{split}
\Big(\E\Big\| \int_0^\cdot\<Z^{i,N}_s, \d
M_s^i\>\Big\|_{\8,t}^p\Big)^{\ff{1}{p}}
\le\ff{\epsilon}{2}(\E\|Z^{i,N}\|^{2p}_{\8,t})^{\ff{1}{p}}+\frac{C}{2\epsilon}\Lambda_t^{i,p}.
\end{split}
\end{equation*}
Thus, \eqref{E5} follows from \eqref{P3} and Gr\"{o}nwall's inequality.
\end{proof}

\subsection{Proof of Proposition \ref{lemma:POCTime}}\label{Sec:POCTime}
\begin{proof}
\noindent
From the definitions of $Y_{t_{n+1}}^{i}$ and $Y_{t_{n+1}}^{i,N}$, we obtain directly
\begin{align*}
& \left|  Y_{t_{n+1}}^{i,N} - Y_{t_{n+1}}^{i} \right|^2 \\
& = \Big| Y_{t_{n}}^{i,N} - Y_{t_{n}}^{i} + b(Y_{t_n}^{i,N},\mu_{t_n}^{Y^{\cdot,N}}) \delta - b(Y_{t_n}^{i}, \mathscr{L}_{Y_{t_n}^{i}}) \delta + \left(\sigma(Y_{t_n}^{i,N},\mu_{t_n}^{Y^{\cdot,N}}) - \sigma(Y_{t_n}^{i}, \mathscr{L}_{Y_{t_n}^{i}}) \right) \Delta W_n^{i} \\
& \quad + \int_{t_n}^{t_{n+1}} \nn \sigma(Y_{t_n}^{i,N},\mu_{t_n}^{Y^{\cdot,N}}) \sigma(Y_{t_n}^{i,N}, \mu_{t_n}^{Y^{\cdot,N}})  \int_{t_n}^{s} \mathrm{d}W^{i}_u  \mathrm{d}W^{i}_s  - \int_{t_n}^{t_{n+1}} \nn \sigma(Y_{t_n}^{i}, \mathscr{L}_{Y_{t_n}^{i}}) \sigma(Y_{t_n}^{i}, \mathscr{L}_{Y_{t_n}^{i}})  \int_{t_n}^{s} \mathrm{d}W_u^{i}  \mathrm{d}W_s^{i} \\
& \quad + \int_{t_n}^{t_{n+1}} \frac{1}{N}\sum_{j = 1}^N  D^{L} \sigma(Y_{t_n}^{i,N}, \mu_{t_n}^{Y^{\cdot,N}})(Y_{t_n}^{j,N}) \sigma(Y_{t_n}^{j,N}, \mu_{t_n}^{Y^{\cdot,N}}) \int_{t_n}^{s} \mathrm{d}W^{j}_u  \mathrm{d}W^{i}_s \Big|^2.
\end{align*} 
Squaring out the right side of the previous equality and taking the expectation on both sides gives 
\begin{align*}
& \mathbb{E} \left|  Y_{t_{n+1}}^{i,N} - Y_{t_{n+1}}^{i} \right|^2 \\
& \leq \mathbb{E} \left|  Y_{t_{n}}^{i,N} - Y_{t_{n}}^{i} \right|^2 + 4\mathbb{E} \Big| b(Y_{t_n}^{i,N},\mu_{t_n}^{Y^{\cdot,N}}) \delta - b(Y_{t_n}^{i}, \mathscr{L}_{Y_{t_n}^{i}}) \delta \Big|^2 + 4 \mathbb{E} \Big| \left(\sigma(Y_{t_n}^{i,N},\mu_{t_n}^{Y^{\cdot,N}}) - \sigma(Y_{t_n}^{i}, \mathscr{L}_{Y_{t_n}^{i}}) \right) \Delta W_n^{i} \Big|^2 \\
& \quad + 4\mathbb{E} \Big| \int_{t_n}^{t_{n+1}} \nn \sigma(Y_{t_n}^{i,N},\mu_{t_n}^{Y^{\cdot,N}}) \sigma(Y_{t_n}^{i,N}, \mu_{t_n}^{Y^{\cdot,N}})  \int_{t_n}^{s} \mathrm{d}W^{i}_u  \mathrm{d}W^{i}_s  - \int_{t_n}^{t_{n+1}} \nn \sigma(Y_{t_n}^{i}, \mathscr{L}_{Y_{t_n}^{i}}) \sigma(Y_{t_n}^{i}, \mathscr{L}_{Y_{t_n}^{i}})  \int_{t_n}^{s} \mathrm{d}W_u^{i}  \mathrm{d}W_s^{i} \Big|^2 \\
& \quad + 4\mathbb{E} \Big| \int_{t_n}^{t_{n+1}} \frac{1}{N}\sum_{j = 1}^N  D^{L} \sigma(Y_{t_n}^{i,N}, \mu_{t_n}^{Y^{\cdot,N}})(Y_{t_n}^{j,N}) \sigma(Y_{t_n}^{j,N}, \mu_{t_n}^{Y^{\cdot,N}}) \int_{t_n}^{s} \mathrm{d}W^{j}_u  \mathrm{d}W^{i}_s \Big|^2 \\
& \quad + 2 \mathbb{E} \Big \langle  Y_{t_{n}}^{i,N} - Y_{t_{n}}^{i} , b(Y_{t_n}^{i,N},\mu_{t_n}^{Y^{\cdot,N}}) \delta - b(Y_{t_n}^{i}, \mathscr{L}_{Y_{t_n}^{i}}) \delta \Big \rangle =  \mathbb{E} \left|  Y_{t_{n}}^{i,N} - Y_{t_{n}}^{i} \right|^2 + \sum_{i=1}^{5} \Pi_i. 
\end{align*}
In the sequel, we estimate each of these expectations one-by-one. 
 
First, using the imposed Lipschitz assumptions on $b(\cdot,\cdot)$, $\sigma(\cdot,\cdot)$ and $(\nn \sigma)\sigma(\cdot,\cdot)$, we can readily estimate, for some constant $C>0$,
\begin{equation*}
\Pi_1 + \Pi_2 + \Pi_3 \leq C\delta \left(\mathbb{E} \left|  Y_{t_{n}}^{i,N} - Y_{t_{n}}^{i} \right|^2 + \mathbb{E}\mathbb{W}_2 \left(\mathscr{L}_{Y_{t_n}^{i}},  \tilde{\mu}_{t_n}^{\tilde{Y}^{\cdot,N}} \right)^2 \right),
\end{equation*}
where $\tilde{\mu}_{t_n}^{\tilde{Y}^{\cdot,N}}$ is the empirical distribution of the non-interacting particle system $(Y^{i}_{t_n})_{i \in \mathbb{S}_N}$.  
The term $\Pi_4$ can be treated as follows: First, we compute
\begin{align*}
&\mathbb{E} \Big| \int_{t_n}^{t_{n+1}} \frac{1}{N}\sum_{j = 1}^N  D^{L} \sigma(Y_{t_n}^{i,N}, \mu_{t_n}^{Y^{\cdot,N}})(Y_{t_n}^{j,N}) \sigma(Y_{t_n}^{j,N}, \mu_{t_n}^{Y^{\cdot,N}}) \int_{t_n}^{s} \mathrm{d}W^{j}_u  \mathrm{d}W^{i}_s \Big|^2 \\
& \leq 2 \mathbb{E} \Big| \int_{t_n}^{t_{n+1}} \frac{1}{N}\sum_{j = 1}^N  \left( D^{L} \sigma(Y_{t_n}^{i,N}, \mu_{t_n}^{Y^{\cdot,N}})(Y_{t_n}^{j,N}) \sigma(Y_{t_n}^{j,N}, \mu_{t_n}^{Y^{\cdot,N}}) - D^{L} \sigma(Y_{t_n}^{j}, \mathscr{L}_{Y^{j}_{t_n}})(Y_{t_n}^{j}) \sigma(Y_{t_n}^{j}, \mathscr{L}_{Y^{j}_{t_n}})  \right) \int_{t_n}^{s} \mathrm{d}W^{j}_u  \mathrm{d}W^{i}_s \Big|^2 \\
& \quad +2 \mathbb{E} \Big| \int_{t_n}^{t_{n+1}} \frac{1}{N}\sum_{j = 1}^N  D^{L} \sigma(Y_{t_n}^{j}, \mathscr{L}_{Y^{j}_{t_n}})(Y_{t_n}^{j}) \sigma(Y_{t_n}^{j}, \mathscr{L}_{Y^{j}_{t_n}}) \int_{t_n}^{s} \mathrm{d}W^{j}_u  \mathrm{d}W^{i}_s \Big|^2 \\
& =: \Pi_{41} + \Pi_{42}.
\end{align*}
Also, due to assumption ({\bf AA}$_{\sigma}^1$) (2), we obtain 
\begin{align*}
\Pi_{41} \leq C \delta^2 \left(\mathbb{E} \left|  Y_{t_{n}}^{i,N} - Y_{t_{n}}^{i} \right|^2  + \mathbb{E}\mathbb{W}_2 \left(\mathscr{L}_{Y_{t_n}^{i}},  \tilde{\mu}_{t_n}^{\tilde{Y}^{\cdot,N}} \right)^2  \right).
\end{align*}
Further, for ease of notation, we restrict the analysis to $d=1$ for the subsequent analysis: 
\begin{align*}
& \Pi_{42} \leq C \mathbb{E} \left(\int_{t_n}^{t_{n+1}} \frac{1}{N}\sum_{j = 1}^N  D^{L} \sigma(Y_{t_n}^{j}, \mathscr{L}_{Y^{j}_{t_n}})(Y_{t_n}^{j}) \sigma(Y_{t_n}^{j}, \mathscr{L}_{Y^{j}_{t_n}}) \int_{t_n}^{s} \mathrm{d}W^{j}_u  \mathrm{d}W^{i}_s \right)^2 \\
& \leq C \frac{1}{N^2}\sum_{j,k =1 }^N  \mathbb{E} \Bigg[ \left(\int_{t_n}^{t_{n+1}} D^{L} \sigma(Y_{t_n}^{j}, \mathscr{L}_{Y^{j}_{t_n}})(Y_{t_n}^{j}) \sigma(Y_{t_n}^{j}, \mathscr{L}_{Y^{j}_{t_n}}) \int_{t_n}^{s} \mathrm{d}W^{j}_u  \mathrm{d}W^{i}_s \right) \times \\
& \quad  \times \left( \int_{t_n}^{t_{n+1}}  D^{L} \sigma(Y_{t_n}^{k}, \mathscr{L}_{Y^{k}_{t_n}})(Y_{t_n}^{k}) \sigma(Y_{t_n}^{k}, \mathscr{L}_{Y^{k}_{t_n}}) \int_{t_n}^{s} \mathrm{d}W^{k}_u  \mathrm{d}W^{i}_s  \right) \Bigg] \\
& =  C \frac{1}{N^2}\sum_{j,k =1 }^N  \mathbb{E} \Bigg[ \int_{t_n}^{t_{n+1}} \left( D^{L} \sigma(Y_{t_n}^{j}, \mathscr{L}_{Y^{j}_{t_n}})(Y_{t_n}^{j}) \sigma(Y_{t_n}^{j}, \mathscr{L}_{Y^{j}_{t_n}}) \int_{t_n}^{s} \mathrm{d}W^{j}_u \right) \times \\
& \quad  \times \left( D^{L} \sigma(Y_{t_n}^{k}, \mathscr{L}_{Y^{k}_{t_n}})(Y_{t_n}^{k}) \sigma(Y_{t_n}^{k}, \mathscr{L}_{Y^{k}_{t_n}}) \int_{t_n}^{s} \mathrm{d}W^{k}_u  \right) \mathrm{d}s \Bigg] \\
& \leq C \frac{\delta^2}{N},
\end{align*}
where in the last display, we employed the independence and moment stability of the random variables $(Y^{i}_{t_n})_{i \in \mathbb{S}_N}$ and the Brownian motions, respectively.
Hence, using above estimates for $\Pi_{41}$ and $\Pi_{42}$, we get
\begin{equation*}
\Pi_4 \leq C \delta \left(\frac{\delta}{N} +  \mathbb{E} \left|  Y_{t_{n}}^{i,N} - Y_{t_{n}}^{i} \right|^2  + \mathbb{E}\mathbb{W}_2 \left(\mathscr{L}_{Y_{t_n}^{i}},  \tilde{\mu}_{t_n}^{\tilde{Y}^{\cdot,N}} \right)^2  \right).
\end{equation*}
Next, we compute, using the fact that $b(\cdot,\cdot)$ is Lipschitz continuous in both components,
\begin{align*}
\Pi_5 & = 2 \mathbb{E} \Big \langle  Y_{t_{n}}^{i,N} - Y_{t_{n}}^{i} , b(Y_{t_n}^{i,N},\mu_{t_n}^{Y^{\cdot,N}}) \delta - b(Y_{t_n}^{i}, \mathscr{L}_{Y_{t_n}^{i}}) \delta \Big \rangle \\
& \leq C\delta \left( \mathbb{E} \left|  Y_{t_{n}}^{i,N} - Y_{t_{n}}^{i} \right|^2 + \mathbb{E}\mathbb{W}_2 \left(\mathscr{L}_{Y_{t_n}^{i}},  \tilde{\mu}_{t_n}^{\tilde{Y}^{\cdot,N}} \right)^2 \right).
\end{align*}
Hence, putting everything together, we obtain
\begin{align*}
\mathbb{E} \left|  Y_{t_{n+1}}^{i,N} - Y_{t_{n}}^{i} \right|^2  & \leq \mathbb{E} \left|  Y_{t_{n}}^{i,N} - Y_{t_{n}}^{i} \right|^2 +   C\delta \left(\mathbb{E} \left|  Y_{t_{n}}^{i,N} - Y_{t_{n}}^{i} \right|^2 + \frac{\delta}{N} + \mathbb{E}\mathbb{W}_2 \left(\mathscr{L}_{Y_{t_n}^{i}},  \tilde{\mu}_{t_n}^{\tilde{Y}^{\cdot,N}} \right)^2 \right)  \\
& \leq (1 + C \delta) \mathbb{E} \left|  Y_{t_{n}}^{i,N} - Y_{t_{n}}^{i} \right|^2 + C \delta \phi(N),
\end{align*}
where $\phi(N)$ is defined in (\ref{D1}). The constant $C$ is independent of the present time-step $n \in \lbrace 0, \ldots, M \rbrace$. 
Consequently, recalling that $ Y_{t_{0}}^{i,N} = Y_{t_{0}}^{i}$, one can inductively prove that $\mathbb{E} \left|  Y_{t_{n}}^{i,N} - Y_{t_{n}}^{i} \right|^2 \leq C\phi(N)$, for any $n \in \lbrace 0, \ldots, M \rbrace$.
\end{proof}

\section{Numerical results and implementation details}\label{Section:Sec4}
We now present a number of numerical tests to illustrate our theoretical results.
For the following Examples 1 to 5, we employ the particle method to approximate the law $\mathscr{L}_{Y_{t_n}}$ at each time-step $t_n$, $n \in \lbrace 0, \ldots, M \rbrace$, by its empirical distribution. For our numerical experiments we used $N=10^4$ (and $N=10^3$ in Example 3), unless otherwise stated.

As we do not know the exact solution in the considered examples, the strong convergence with respect to the number of time-steps is assessed by comparing two solutions computed on a fine and coarse time grid, respectively, using the same samples of Brownian motion. 
Specifically, we compute the root-mean-square error (RMSE) between the numerical solution $Y_T^{\cdot,N,l}$ at a level $l$ of the time-discretisation with $2^lT$ time steps and $N$ particles,
and the solution at level $l-1$, at the final time $T=1$,
\begin{equation*}
\text{RMSE}:= \sqrt{\frac{1}{N} \sum_{i=1}^{N} \left(Y_T^{i,N,l} - Y_T^{i,N,l-1} \right)^2},
\end{equation*}
where $i$ refers to the $i$-th particle, $i \in \mathbb{S}_N$.

In Examples 1 to 5 below, we study the convergence order in $\delta$, the relevance of the $L$-derivative for a small number of particles and the asymptotic behaviour of the $L$-derivative terms of the Milstein scheme, and propagation of chaos.

\subsection{Examples}
In the examples presented below, we choose the parameter values $\sigma = 1.5$, $c=0.5$ and set $X_0=1$. 
These non-globally Lipschitz SDEs satisfy all assumptions needed to guarantee a unique strong solution. 


While Examples 1 to 3 have a linear measure dependence, i.e.,  $\sigma$ (or $b$) are of the form $\sigma(x,\mu) = \int_{\R} \tilde{\sigma}(x,x') \, \mu(\mathrm{d}x')$, for some function $\tilde{\sigma}: \R \times \R \to \R$, Examples 4 and 5 have coefficients with non-linear measure dependence.
\\

\noindent
\textbf{Example 1:} \qquad
$
    \mathrm{d}X_t = \left( \frac{\sigma^2}{2}X_t -X^3_t + c \mathbb{E}X_t \right) \d t + \mathbb{E} X_t  \d W_t.
$
\medskip \\
\noindent
\textbf{Example 2:} \qquad
$
    \mathrm{d}X_t = \left( \frac{\sigma^2}{2}X_t -X^3_t + c \mathbb{E} X_t \right) \d t + X_t  \d W_t.
$
\medskip \\
\noindent
\textbf{Example 3:} \qquad
$
    \mathrm{d}X_t = \left( \frac{\sigma^2}{2}X_t -X^3_t + \int_{\RR} \sin \left(X_t-y \right) \mathscr{L}_{X_t}(\mathrm{d}y) \right)\mathrm{d}t + \left(X_t + \int_{\RR} \sin \left(X_t-y \right) \mathscr{L}_{X_t}(\mathrm{d}y) \right)  \d W_t. 
$ \smallskip \\
\noindent
\textbf{Example 4:} \qquad
$
    \mathrm{d}X_t = \left( \frac{\sigma^2}{2}X_t -X^3_t + \mathbb{E} X_t \right)\mathrm{d}t + \int_{\RR} \int_{\RR} \sin (x+y) \mathscr{L}_{X_t}(\mathrm{d}y)  \mathscr{L}_{X_t}(\mathrm{d}x) \d W_t. 
$ \smallskip \\
\textbf{Example 5:} \qquad
 $
    \mathrm{d}X_t = \left( \frac{\sigma^2}{2}X_t -X^3_t + \mathbb{E} X_t \right)\mathrm{d}t + e^{-\mathbb{V}\left[ \frac{X_t}{1+X^2_t} \right]} \d W_t. 
$ \medskip \\
Note that not all model assumptions needed to guarantee moment stability for Scheme 1 are satisfied in the case of Examples 1, 2, 4 and 5, as ({\bf A}$_b^{4}$) and/or ({\bf A}$_{\sigma}^{4}$) are violated. However, as illustrated numerically, it seems reasonable that the established theory is also valid in a more general setting. Furthermore, we remark that in Example 1 the diffusion coefficient only contains an expectation and does not explicitly depend on the current state $X_t$ (i.e., the derivatives with respect to the state component vanishes). Similarly, in Examples 4 and 5 we also do not have an explicit state dependence. Hence, we expect the tamed Euler scheme to converge with order 1. This is confirmed numerically below as well. The drift $b_\dd$ of the tamed Euler scheme is chosen as in Scheme 1 for the Milstein scheme.


\subsection{Time-stepping schemes and $L$-derivative}

We first use the two different tamed Milstein schemes (i.e., Scheme 1 and Scheme 2) without the $L$-derivative terms for Examples 1 to 3.
For all these examples, we observe in Fig.\ \ref{fig11} and Fig.\ \ref{fig111} strong convergence of order 1.

From these numerical results, it is also apparent that Scheme 1 consistently outperforms Scheme 2 in terms of accuracy. We further note that the strong convergence order for Scheme 2 seems to be observable for extremely fine time-grids (i.e., $M \geq 2^{12}$) only, which might be due to a large implied constant appearing in the strong convergence analysis of 
Remark \ref{rmk:TH2Scheme2}.
\begin{figure}[!h]
\begin{subfigure}[b]{0.48\textwidth}
\includegraphics[width=\textwidth]{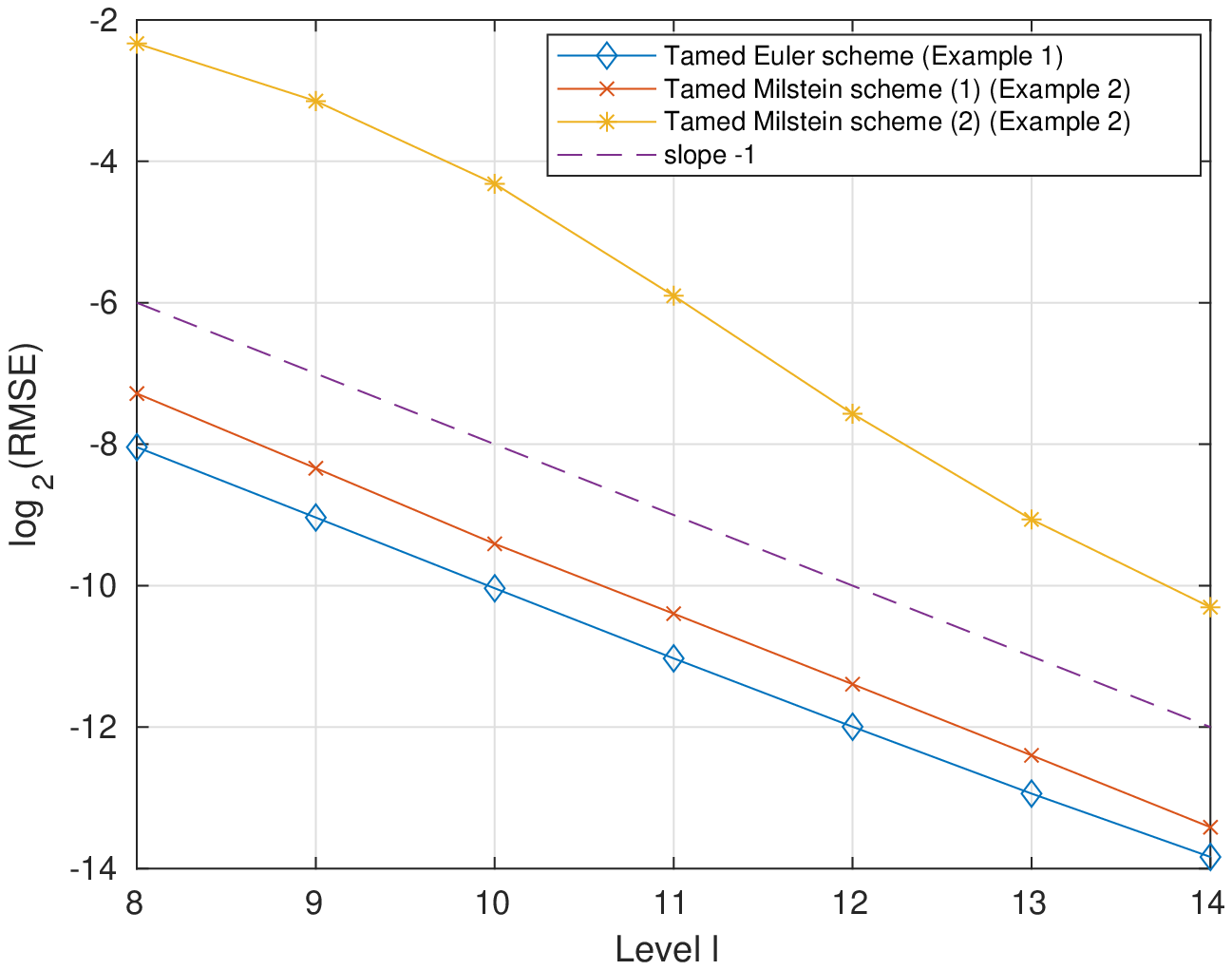}
\caption{}
\label{fig11}
\end{subfigure}
\begin{subfigure}[b]{0.48\textwidth}
\includegraphics[width=\textwidth]{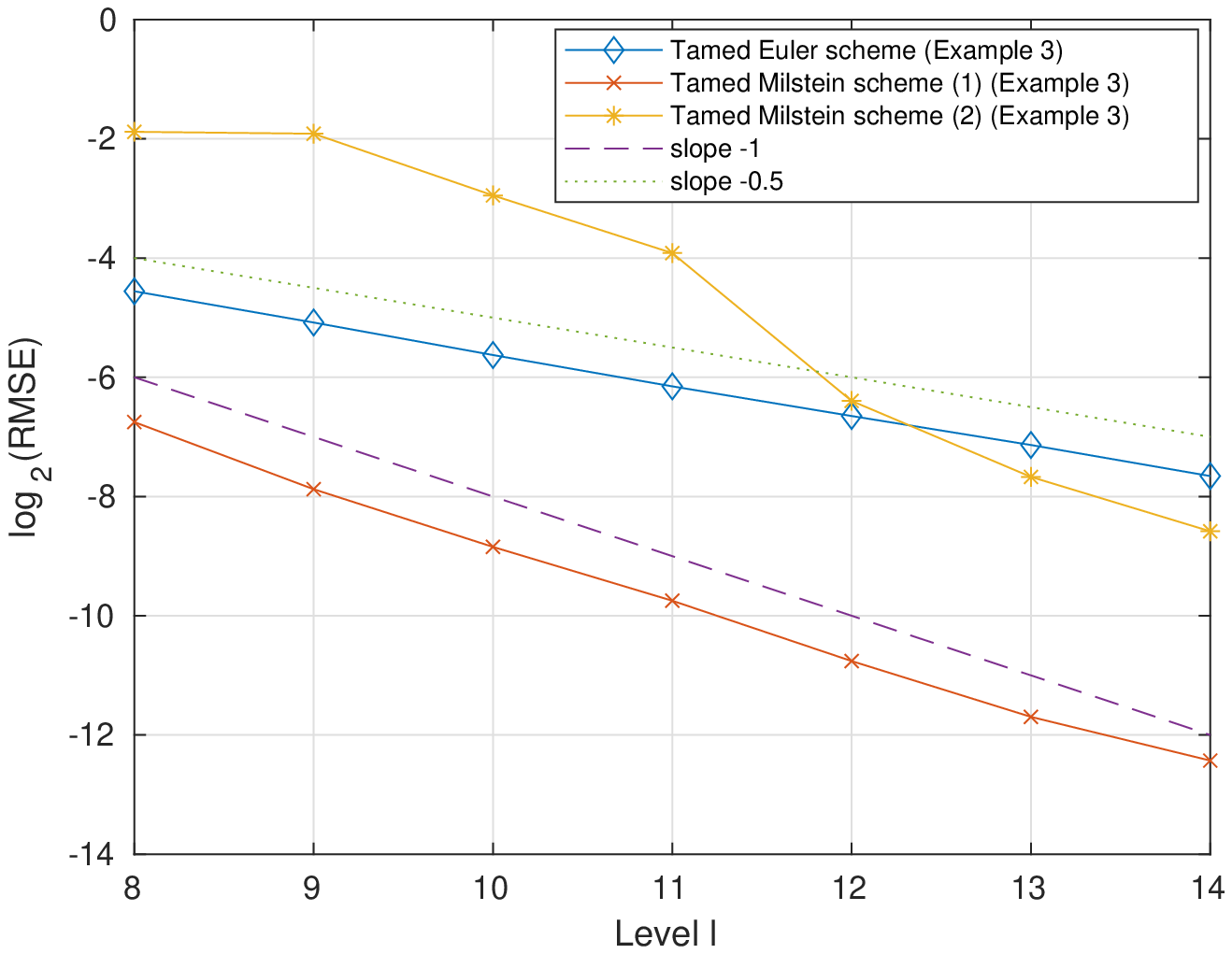}
\caption{}
\label{fig111}
\end{subfigure}
\caption{Numerical illustration for the strong convergence of Example 1, Example 2 (both left), and Example 3 (right) using the tamed Euler scheme and both tamed Milstein schemes, Scheme 1 and Scheme 2 (without the $L$-derivative terms), respectively. In Example 1, the tamed Milstein scheme without the $L$-derivative terms then coincides with the tamed Euler scheme.}
\end{figure}

Note that in the above examples the $L$-derivative of the diffusion term can be explicitly computed, which allows us to compute the term (\ref{eq:LionsTerm}). In Example 1, the $L$-derivative of the diffusion term is simply $D^L\si(x,\mu)(y) =1$, for all $y \in \RR$ (see \cite[Example 1 on page 385]{CD}). Similarly, the $L$-derivative of the diffusion term appearing in Example 3 can be computed as  $D^L\si(x,\mu)(y) = \partial_{y} \sin(x-y) = - \cos(x-y)$, for all $y \in \RR$ (see \cite[Example 4 on page 389]{CD}), and for Example 4 as $D^L\si(x,\mu)(y) = 2 \int_{\R} \cos(x+y) \, \mu(\mathrm{d}x)$,
for all $y \in \RR$ (see \cite[Example 3 on page 387]{CD}).
Finally, for Example 5,
the chain rule shows that the $L$-derivative of $f:\mathscr{P}_2(\R^{d}) \to \R$ defined as 
\begin{equation*}
f(\mu) := \exp\Big( -\int_{\RR} g^2(x) \mu(\mathrm{d}x) + \Big(\int_{\RR} g(x) \mu(\mathrm{d}x) \Big)^2 \Big),
\end{equation*}
for some bounded and continuously differentiable function $g:\R \to \R$ (with bounded derivative), is
\begin{align*}
D^{L} f(\mu)(y) = 
  \exp\Big(-\int_{\RR} g^2(x) \mu(\mathrm{d}x) + \Big(\int_{\RR} g(x) \mu(\mathrm{d}x) \Big)^2\Big)
 \Big( -2g(y)g'(y) + 2 g'(y) \int_{\RR} g(x) \mu(\mathrm{d}x) \Big),
\end{align*}
which is then applied with $g(x)= \frac{x}{1+x^2}$.

We now state, for illustration purposes, the (full) tamed Milstein schemes for Examples 1 and 3.
The tamed Milstein scheme (Scheme 1) for Example 1 reads as
\begin{align*}
Y_{t_{n+1}}^{i,N} &= Y_{t_n}^{i,N} + \frac{ \frac{\sigma^2}{2} Y_{t_n}^{i,N} - (Y_{t_n}^{i,N})^3 +  \frac{c}{N} \sum_{j=1}^{N} Y_{t_n}^{j,N}}{1 + \delta \left| \frac{\sigma^2}{2} Y_{t_n}^{i,N} - (Y_{t_n}^{i,N})^3 +  \frac{c}{N} \sum_{j=1}^{N} Y_{t_n}^{j,N} \right| } \delta +  \frac{1}{N} \sum_{j=1}^{N} Y_{t_n}^{j,N} \Delta W_n^{i} \nonumber \\
& \quad + \frac{1}{N} \sum_{l=1}^{N}  Y_{t_n}^{l,N} \frac{1}{N}\sum_{j = 1}^N \int_{t_n}^{t_{n+1}}  \int_{t_n}^{s} \mathrm{d}W^{j}_u  \mathrm{d}W^{i}_s.
\end{align*}

The tamed Milstein scheme (Scheme 1) for Example 3 has the form
\begin{align*}
Y_{t_{n+1}}^{i,N} &= Y_{t_n}^{i,N} + \frac{ \frac{\sigma^2}{2} Y_{t_n}^{i,N} - (Y_{t_n}^{i,N})^3 +  \frac{1}{N} \sum_{j=1}^{N} \sin \left(  Y_{t_n}^{i,N} -  Y_{t_n}^{j,N} \right)}{1 + \delta \left| \frac{\sigma^2}{2} Y_{t_n}^{i,N} - (Y_{t_n}^{i,N})^3 +  \frac{1}{N} \sum_{j=1}^{N} \sin \left(  Y_{t_n}^{i,N} -  Y_{t_n}^{j,N} \right) \right| } \delta  \\
& \quad  +  \left(Y_{t_n}^{i,N} +  \frac{1}{N} \sum_{j=1}^{N} \sin \left(  Y_{t_n}^{i,N} -  Y_{t_n}^{j,N} \right) \right) \Delta W_n^{i} \nonumber \\
& \quad+ \left(1 + \frac{1}{N} \sum_{j=1}^{N} \cos \left(  Y_{t_n}^{i,N} -  Y_{t_n}^{j,N} \right) \right) \left(Y_{t_n}^{i,N} + \frac{1}{N} \sum_{j=1}^{N} \sin \left(  Y_{t_n}^{i,N} -  Y_{t_n}^{j,N} \right) \right) \frac{(\Delta W_n^{i})^2-\delta}{2} \nonumber \\
& \quad - \frac{1}{N}\sum_{j = 1}^N  \cos \left(  Y_{t_n}^{i,N} -  Y_{t_n}^{j,N} \right) \left( Y_{t_n}^{j,N} + \frac{1}{N} \sum_{l=1}^{N} \sin \left(  Y_{t_n}^{j,N} -  Y_{t_n}^{l,N} \right)  \right) \int_{t_n}^{t_{n+1}}  \int_{t_n}^{s} \mathrm{d}W^{j}_u  \mathrm{d}W^{i}_s.
\end{align*}

To illustrate the effect of the $L$-derivative,
in Fig.\ \ref{fig11LD}, we compare for Example 3 the strong convergence rate of the tamed Milstein scheme (Scheme 1) without the terms involving the $L$-derivative, with a full-tamed Milstein scheme, i.e., including the $L$-derivative terms, for $N=20$. We employ the approximation to  the L\'{e}vy areas proposed in \cite{WSON}. The complexity for computing the L\'{e}vy areas for all time-steps is $\mathcal{O}(N^2M^{3/2})$, where the additional factor $M^{1/2}$ comes from the choice for the truncation level of the so-called Karhunen-Lo\`{e}ve expansion of certain Brownian Bridge processes.

\begin{figure}[!h]
\begin{subfigure}[b]{0.43\textwidth}
\includegraphics[width=\textwidth]{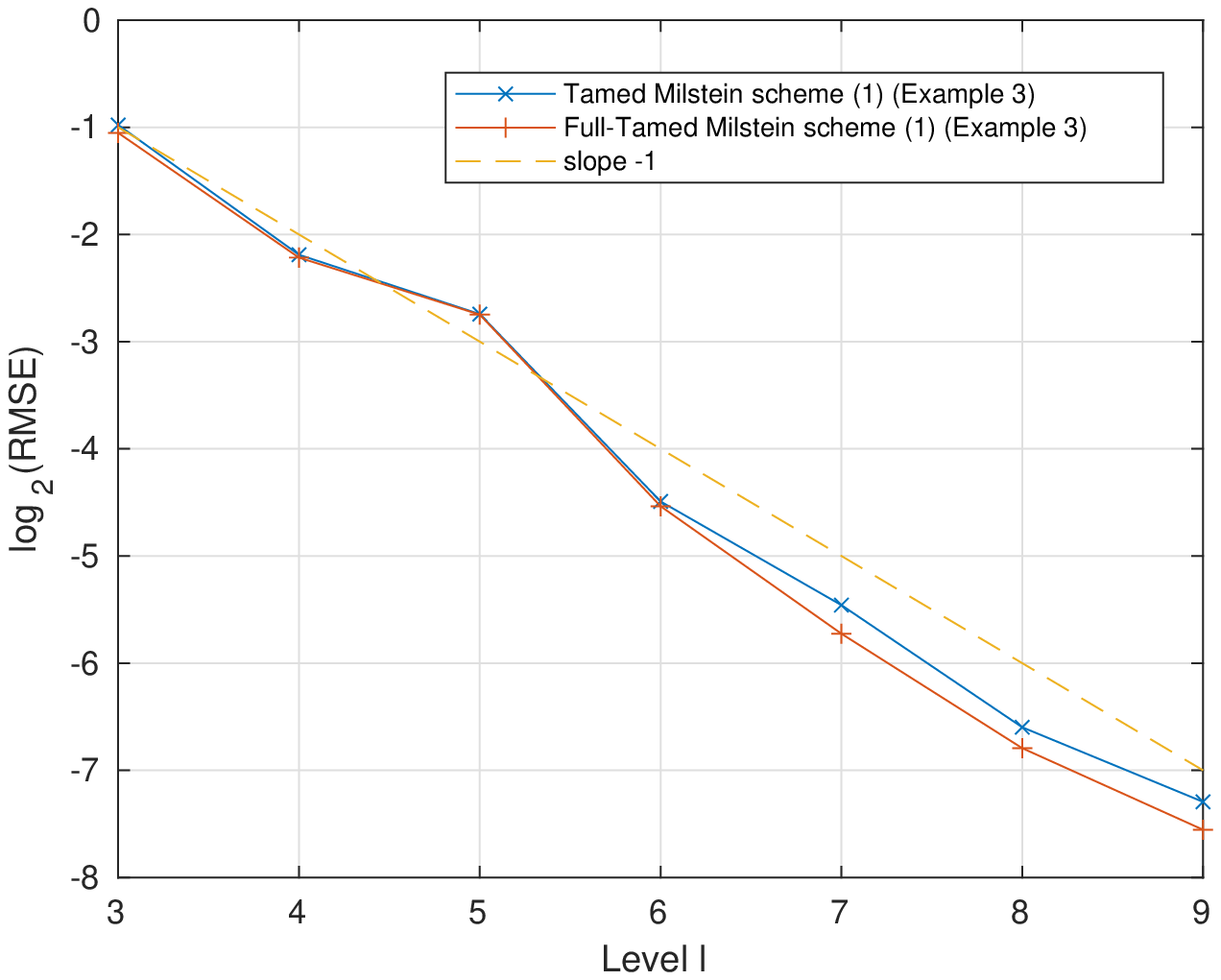}
\caption{}
\label{fig11LD}
\end{subfigure}
\begin{subfigure}[b]{0.43\textwidth}
\includegraphics[width=\textwidth]{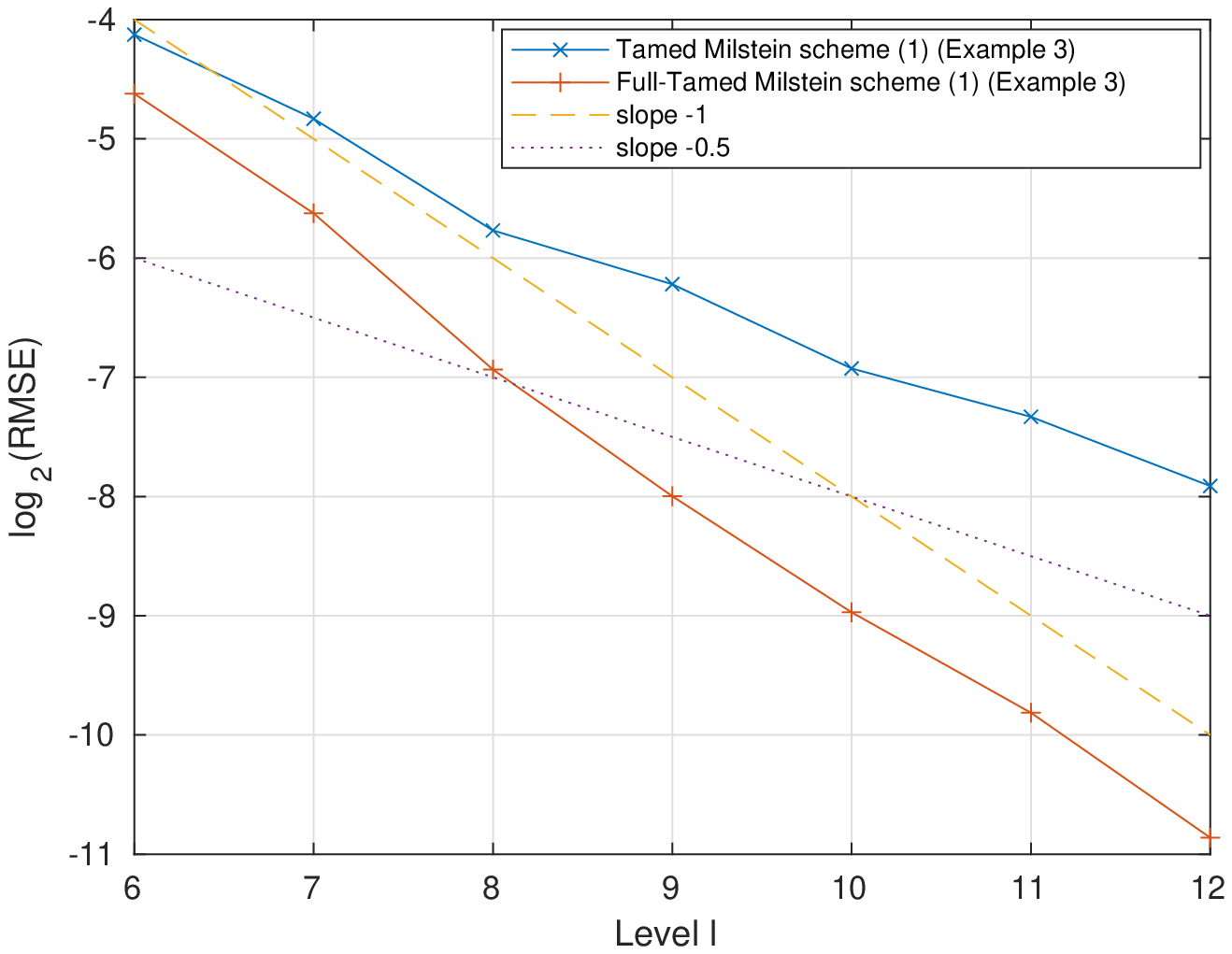}
\caption{}
\label{fig11aLD}
\end{subfigure}
\caption{Numerical illustration of the strong convergence in Example 3 for the tamed Milstein scheme (Scheme 1) with and without the $L$-derivative terms, with $N=20$ particles (left) and $N=3$  (right).}
\end{figure}

Both schemes converge strongly with order 1. Although the full-tamed Milstein scheme shows a slightly better accuracy, we note that this additional gain is not significant taking the complexity of computing the L\'{e}vy areas into account.\footnote{Due to the high computational complexity, the time discretisation is coarser than the one chosen in Fig.\ \ref{fig111LD} below.} 
In Fig.\ \ref{fig11aLD} we show results for the same test with $N=3$; also performed for Examples 4 and 5 in Fig.\ \ref{fig1111LD}. This illustrates the necessity of implementing the $L$-derivative terms to achieve the desired strong convergence order of 1 for fixed $N$, since the scheme without the $L$-derivative terms only converges with order $1/2$. 
\begin{figure}[!h]
\begin{subfigure}[b]{0.43\textwidth}
\includegraphics[width=\textwidth]{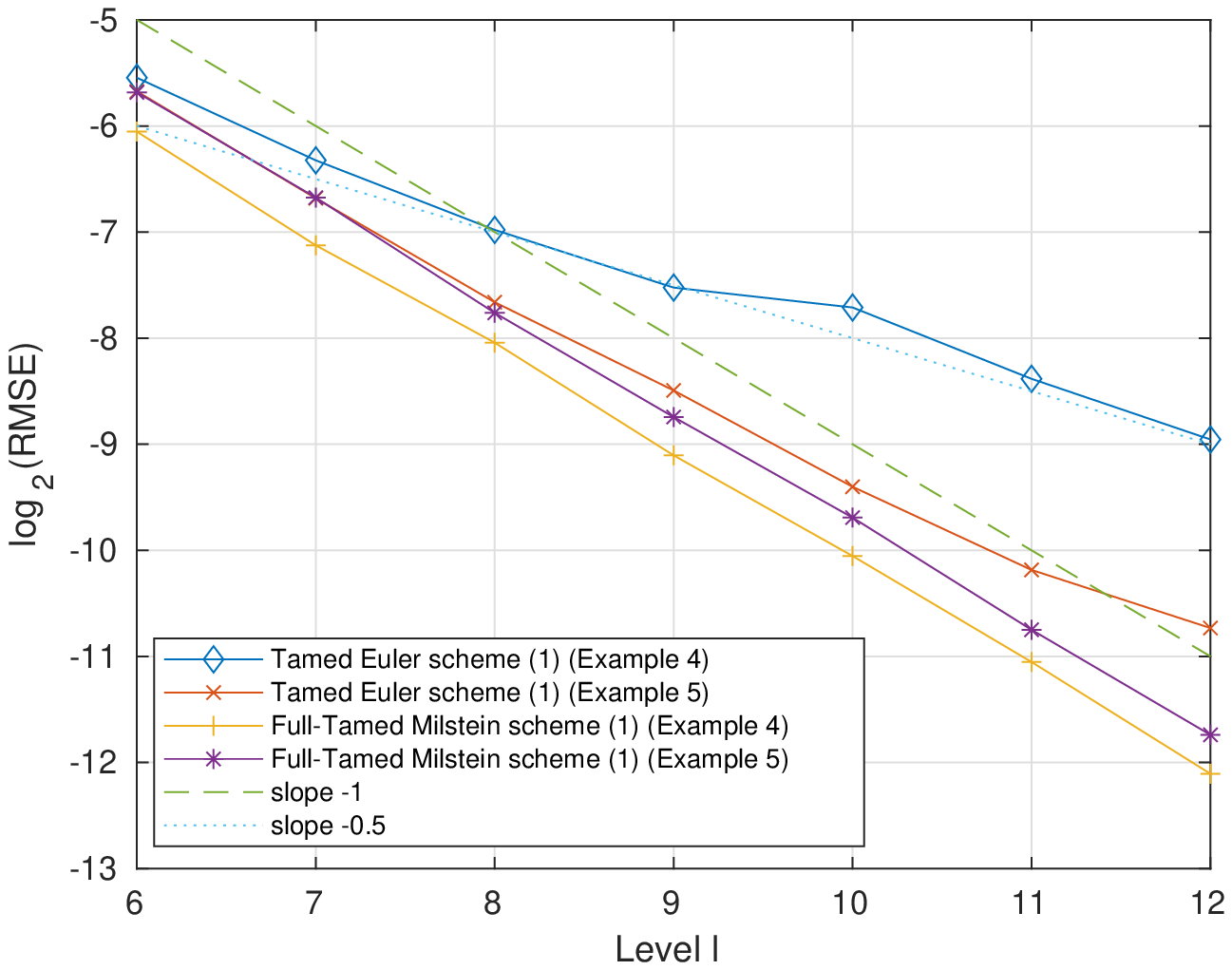}
\caption{}
\label{fig1111LD}
\end{subfigure}
\begin{subfigure}[b]{0.43\textwidth}
\includegraphics[width=\textwidth]{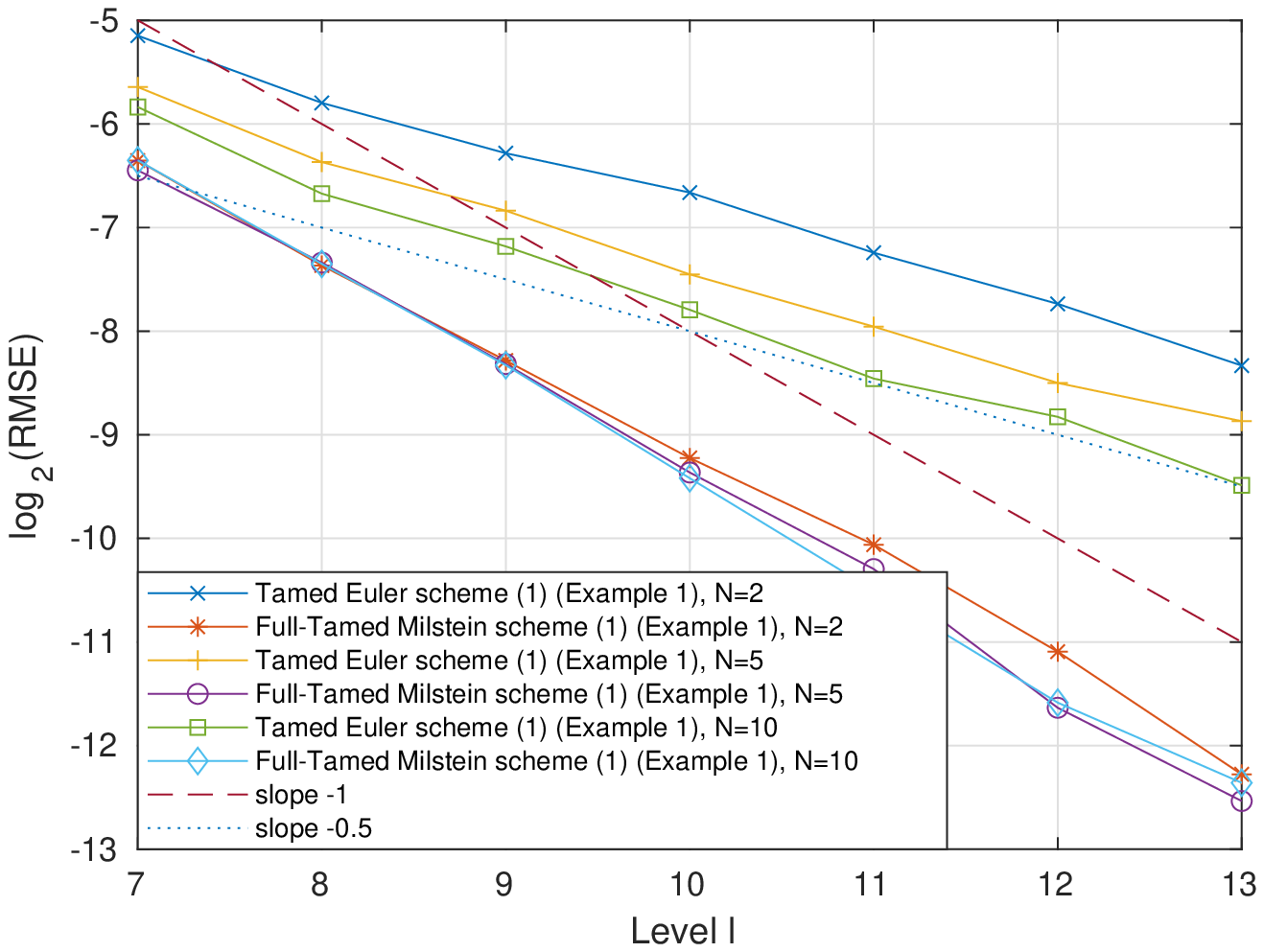}
\caption{}
\label{fig111LD}
\end{subfigure}
\caption{{
Illustration of the effect of the $L$-derivative terms in the tamed Milstein scheme (Scheme 1) for Examples 4 and 5 (left, for $N=3$ particles) and Example 1 (right, for three different choices of $N$).} }
\end{figure}
Fig.\ \ref{fig111LD} explores this further by comparing, in the case of Example 1 (with $c=1$), the strong convergence rate of a tamed Euler scheme with a full-tamed Milstein scheme including the $L$-derivative terms for different small $N$. 

In Fig.\ \ref{fig111LDRevision} and Fig.\ \ref{fig1111LDRevision}, we compute for a given step size $\delta=1/M$ the root-mean-square difference (for several choices of $N$) between a numerical approximation of the particle system associated with the McKean--Vlasov SDE in Examples 4 and 5, obtained using a tamed Euler scheme (i.e., ignoring the $L$-derivative terms) and one obtained by the full-tamed Milstein scheme, at the final time $T=1$. We choose $M=2^{4}, 2^5, 2^6$ for our tests. For a fixed choice of $M$, we then compute the RMSE for $N_l=2^{l}$, 
for $l=2, \ldots, 6$, where 
\begin{equation*}
\text{RMSE}= \sqrt{\frac{1}{N_l} \sum_{i=1}^{N_l} \left(Y_T^{i,N_l,M} - \tilde{Y}_T^{i,N_l,M} \right)^2},
\end{equation*}
where $\tilde{Y}_T^{i,N_l,M}$ denotes the numerical approximation using a tamed Euler scheme of $X$ at time $T$ computed with $N_l$ particles and $M$ time steps, while $Y_T^{i,N_l,M}$ was computed using a full-tamed Milstein scheme.

\begin{figure}[!h]
\begin{subfigure}[b]{0.43\textwidth}
\includegraphics[width=\textwidth]{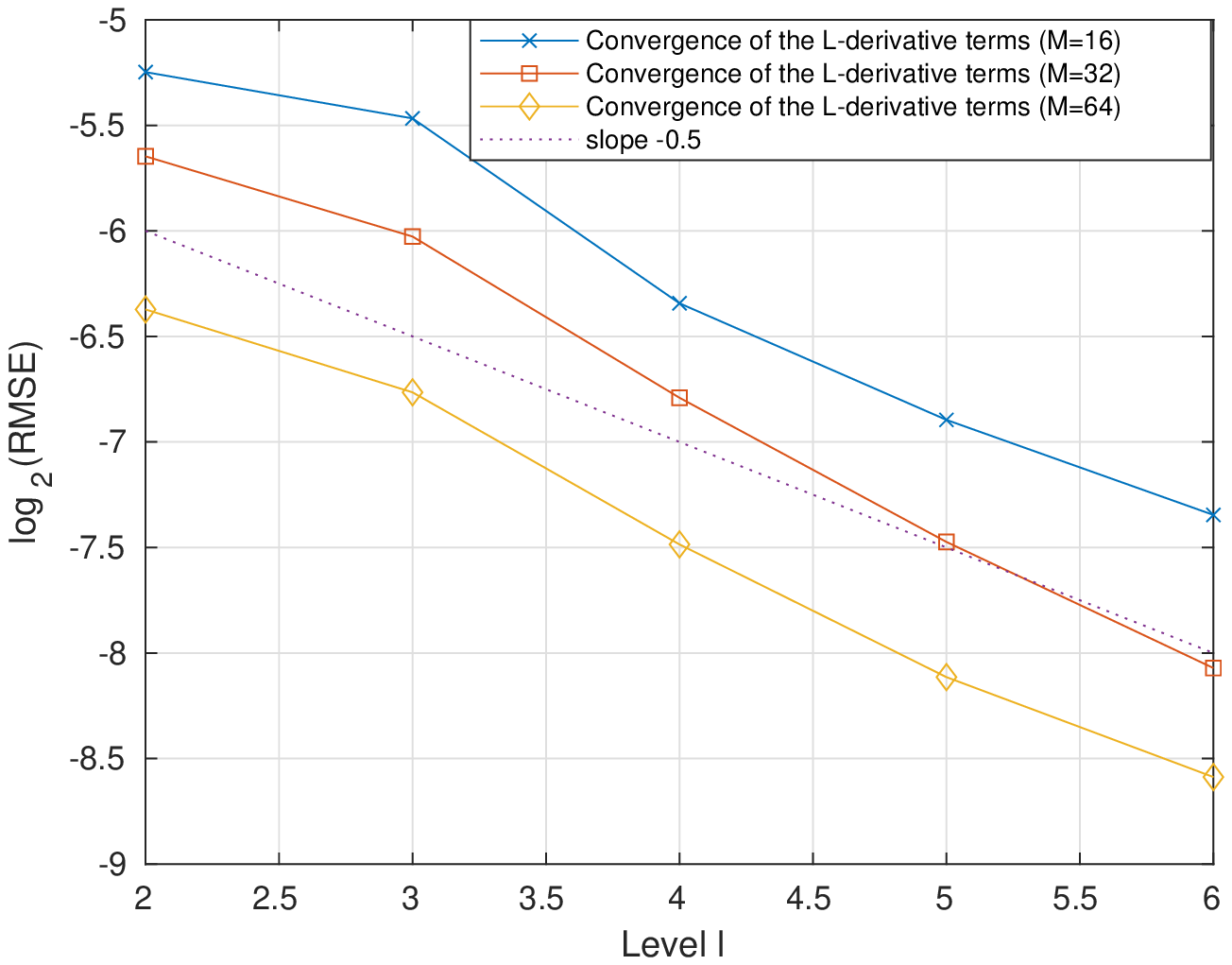}
\caption{}
\label{fig111LDRevision}
\end{subfigure}
\begin{subfigure}[b]{0.43\textwidth}
\includegraphics[width=\textwidth]{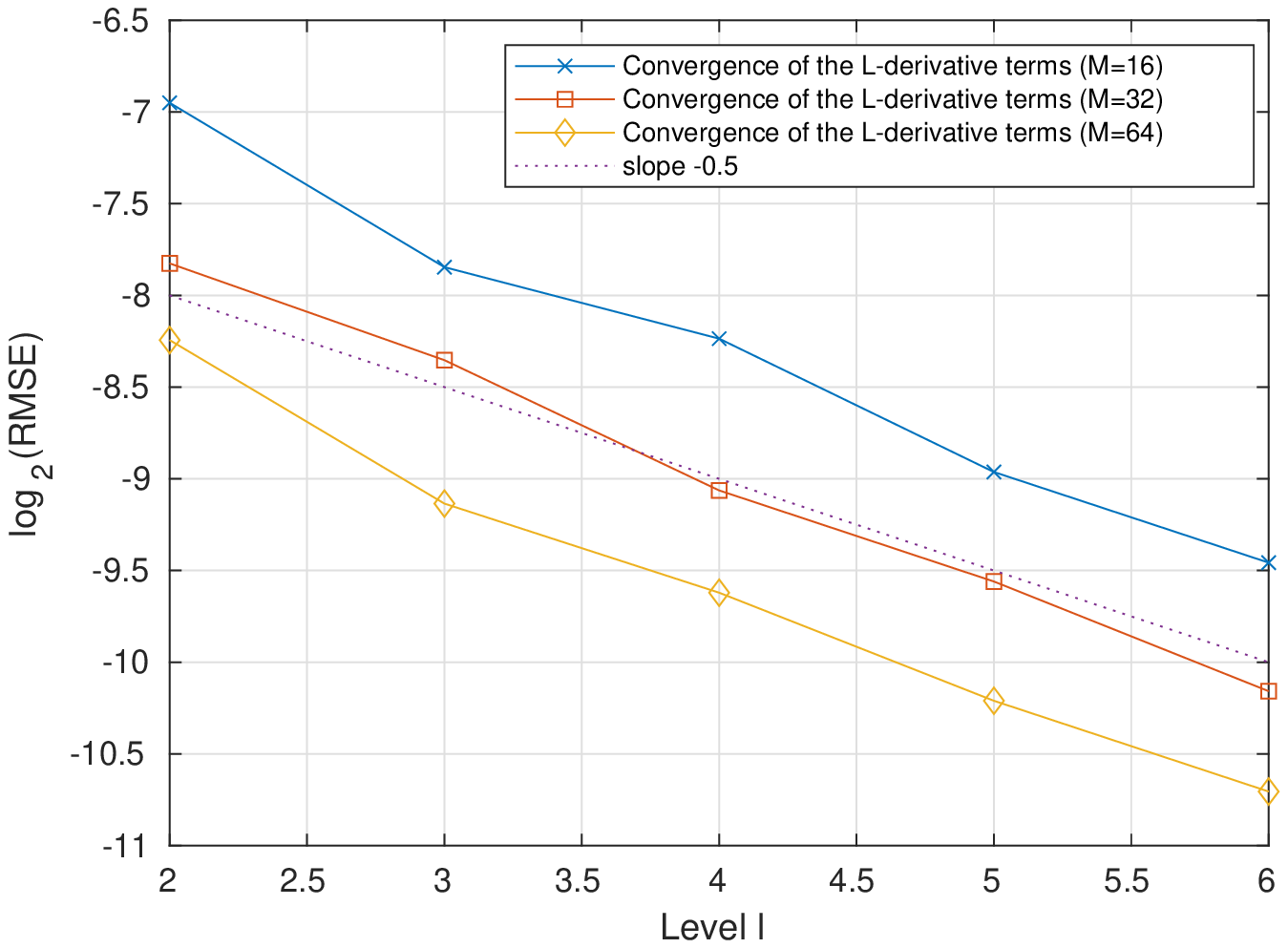}
\caption{}
\label{fig1111LDRevision}
\end{subfigure}
\caption{Convergence of the terms involving the $L$-derivatives with an increasing number of particles for Example 4 (left) and Example 5 (right).}
\end{figure}

We observe a strong convergence rate $1/2$ with respect to the number of particles. The same test is performed for Example 1 in Fig.\ \ref{fig1111LDaRevision}, and we observe a similar convergence behaviour of the $L$-derivative terms as for Examples 4 and 5 above.

\subsection{Propagation of chaos}
Fig.\ \ref{fig1111LDa} depicts the strong propagation of chaos convergence rate for Example 1. For a fixed number of time steps $M=2^6$, we compute for different sizes of the particle system, $N_l =2^l$, for $l=3,\ldots, 7$, a RMSE between two particle systems, 
\begin{equation*} 
\text{RMSE} = \sqrt{\frac{1}{N_l} \sum_{i=1}^{N_l} \left(Y^{i,N_l,M}_T - \tilde{Y}^{i,N_{l},M}_T \right)^2 }, 
\end{equation*}  
where the particle system $\tilde{Y}^{i,N_{l},M}$, $1 \leq i \leq N_l$, is obtained by splitting the set of Brownian motions driving the particle system $Y^{i,N_{l},M}$, $1 \leq i \leq N_l$, in two sets and simulating two independent particle systems, each of size $N_l/2$:
$(\tilde{Y}^{i,N_{l},M,(1)})_{1 \leq i \leq N_l/2}$ is a particle system obtained by using the Brownian motions $(W^{i})_{1 \leq i \leq N_l/2}$ and $(\tilde{Y}^{i,N_{l},M,(2)})_{N_l/2+1 \leq i \leq N_l}$ is obtained from the set $(W^{i})_{N_l/2 + 1 \leq i \leq N_l}$.
In particular, this means that for these two smaller particle systems only $N_l/2$ particles are used to approximate the mean-field term. As a time discretisation scheme, we employ a tamed Euler scheme and the full-tamed Milstein scheme, respectively. We observe a strong convergence order of $1/2$ in terms of number of particles. Corollary \ref{Coro1} only proves a strong convergence order of $1/4$ (see equation (\ref{D1}) for $d=1$). However, recent results (see, e.g., \cite[Theorem 2.4]{LSAT}) suggest that the optimal rate is $1/2$. These results do not apply to our problem description, as they rely on strong regularity assumptions on the coefficients of the underlying McKean--Vlasov SDE (in particular, the drift needs to be globally Lipschitz continuous in the state component).

\begin{figure}[!h]
\begin{subfigure}[b]{0.45\textwidth}
\includegraphics[width=\textwidth]{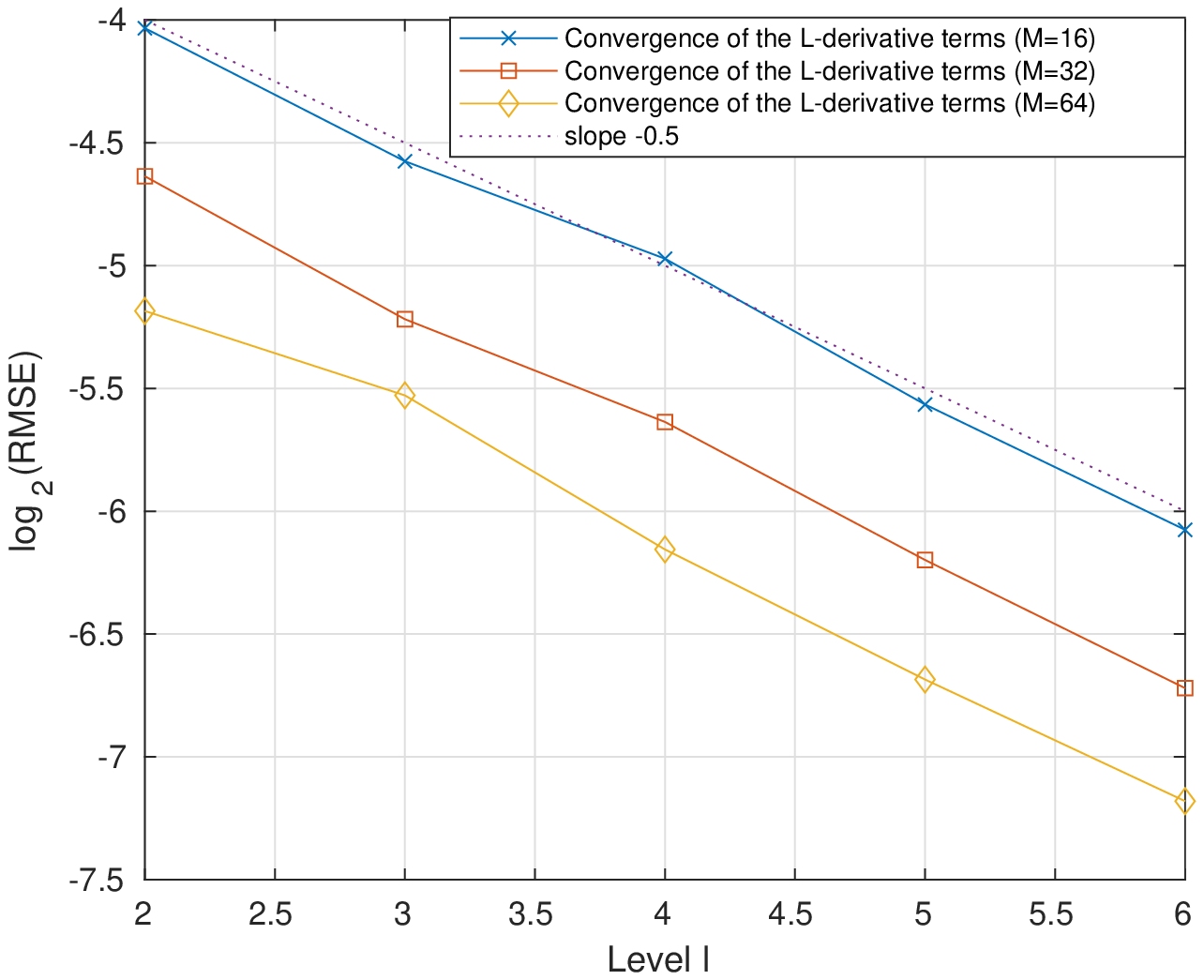}
\caption{}
\label{fig1111LDaRevision}
\end{subfigure}
\begin{subfigure}[b]{0.45\textwidth}
\includegraphics[width=\textwidth]{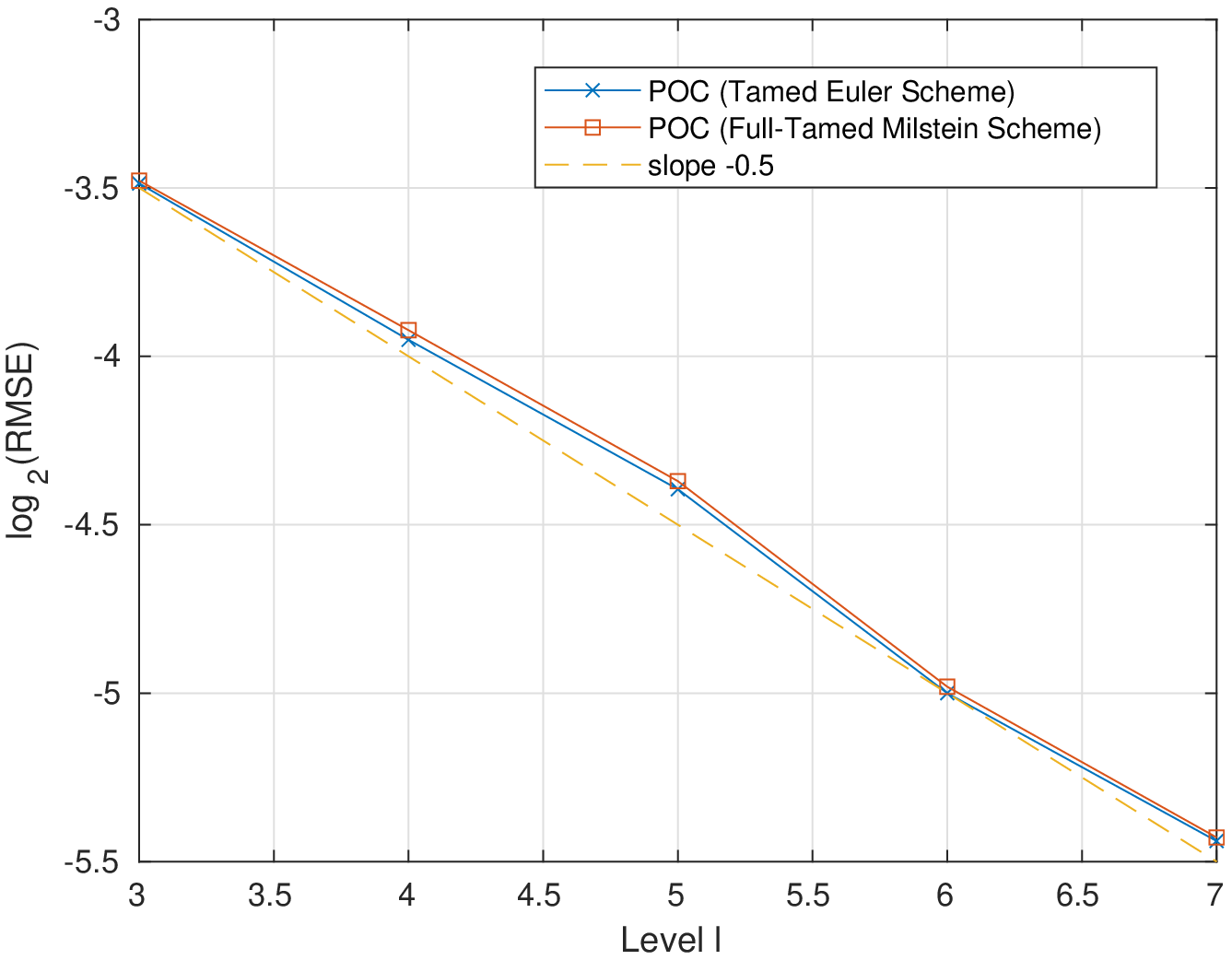}
\caption{}
\label{fig1111LDa}
\end{subfigure}
\caption{Convergence of the terms involving the $L$-derivatives with an increasing number of particles for Example 1 (left). Illustration of the propagation of chaos property (POC) for Example 1 (right).}
\end{figure}

\appendix

\section{Measure derivative}\label{Appendx:A}
We briefly introduce the Lions derivative of a functional $f: \mathscr{P}_2(\mathbb{R}^d) \to \mathbb{R}$, as it will appear in the formulation of the tamed Milstein scheme. For more details about the definition and further results we refer to \cite{PLI} or \cite{BLPR, HSS}. Here, we follow the exposition of \cite{CD}. 
We recall the fact that over an atomless probability space $(\Omega, \F,\mathbb{P})$, for every $\mu \in \mathscr{P}_2(\R^d)$ there is a random variable $X \in L_2(\Omega, \F,\mathbb{P};\mathbb{R}^d)$ such that $\mathscr{L}_X = \mu$, see e.g., \cite[Section 5.2]{CD}.
We will then associate to the function $f$ a lifted function $\tilde{f}$, which allows one to introduce the $L$-derivative as Fr\'{e}chet derivative, by $\tilde{f}(X)=f(\mathscr{L}_X)$, for $X \in L_2(\Omega, \F,\mathbb{P};\mathbb{R}^d)$. 
\begin{defn}
A function $f$ on $\mathscr{P}_2(\mathbb{R}^d)$ is said to be $L$-differentiable at $\mu_0 \in \mathscr{P}_2(\mathbb{R}^d)$ if there exists a random variable $X_0$ with law $\mu_0$, such that the lifted function $\tilde{f}$ is Fr\'{e}chet differentiable at $X_0$, i.e., the Riesz representation theorem implies that there is a ($\mathbb{P}$-almost surely) unique $\Theta \in L_2(\Omega, \F,\mathbb{P};\mathbb{R}^d)$ with
\begin{equation*}
\tilde{f}(X) = \tilde{f}(X_0) + \langle \Theta,  X-X_0 \rangle_{L_2}  + o(\| X-X_0\|_{L_2}), \text{ as } \| X-X_0\|_{L_2} \to 0,
\end{equation*}
with the standard inner product and norm on $L_2(\Omega, \F,\mathbb{P};\mathbb{R}^d)$ and $X \in L_2(\Omega, \F,\mathbb{P};\mathbb{R}^d)$. If $f$ is $L$-differentiable for all $\mu_0 \in \mathscr{P}_2(\mathbb{R}^d)$, then we say that $f$ is $L$-differentiable.
\end{defn} 

It is known (see e.g., \cite[Proposition]{CD}), that there exists a Borel measurable function $\xi: \mathbb{R}^d \to \mathbb{R}^d$, such that $\Theta =  \xi(X_0)$ almost surely, and hence
\begin{equation*}
f(\mathscr{L}_X) = f(\mathscr{L}_{X_0}) + \mathbb{E}\left\langle \xi(X_0), X -X_0 \right \rangle +o(\| X-X_0\|_{L_2}).
\end{equation*}   
Note that $\xi$ only depends on the law of $X_0$, but not on $X_0$ itself.
We define $D^{L}f(\mathscr{L}_{X_0})(y):=\xi(y)$, $y \in \mathbb{R}^d$, as the $L$-derivative of $f$ at $\mu_0$. Observe that $D^{L}f(\mathscr{L}_{X_0})(y)$ is only $\mathscr{L}_{X_0}(\mathrm{d}y)$ almost everywhere uniquely defined. Below and in the analysis of our time-stepping schemes we will always work with a $\mathscr{L}_{X_0}$ version of $D^{L}f(\mathscr{L}_{X_0})(\cdot): \mathbb{R}^{d} \to \R^{d}$. Further, assume for each fixed $y \in \mathbb{R}^d$ that the map $\mathscr{P}_2(\mathbb{R}^d) \ni \mu \mapsto D^{L}f(\mu)(y)$ is continuously $L$-differentiable. Hence, the $L$-derivative of the components $(D^{L}f)_j(\cdot)(y): \mathscr{P}_2(\mathbb{R}^d) \to \mathbb{R}$, $j \in \lbrace 1, \ldots, d \rbrace$, is defined as
\begin{equation*}
(D^{L})^2 f(\mu)(y,y'):= (D^{L}((D^{L}f)_j(\cdot)(y))(\mu,y'))_{j=1, \ldots, d},
\end{equation*}
for $(\mu,y,y') \in \mathscr{P}_2(\mathbb{R}^d) \times \mathbb{R}^d \times \mathbb{R}^d$. For a vector-valued (or matrix-valued) function $f$, these definitions have to be understood componentwise. 

For the strong convergence analysis, we employ a definition describing regularity properties of a function $f: \mathbb{R}^d \times \mathscr{P}_2(\mathbb{R}^d) \to \mathbb{R}$ in terms of the measure derivative, see \cite{CD, CCD}.  
\begin{defn}\label{def:Defin1}
Let $f: \mathbb{R}^d \times \mathscr{P}_2(\mathbb{R}^d) \to \mathbb{R}$ be a given functional. 
\begin{itemize}
\item We say that $f \in C^{2,(1,1)}(\mathbb{R}^d \times \mathscr{P}_2(\mathbb{R}^d))$ if for any $\mu \in \mathscr{P}_2(\mathbb{R}^d)$ $f(\cdot, \mu)$, is twice continuously differentiable and for any $x \in \R^d$ $f(x, \cdot)$ is partially $\mathcal{C}_2$, i.e., for any $\mu \in \mathscr{P}_2(\mathbb{R}^d)$ and $x \in \R^d$, there is a continuous version of the map $\mathbb{R}^{d} \ni y \mapsto D^{L} f(x,\mu)(y)$ such that the derivatives
\begin{align*}
D^{L} f(x,\mu)(y), \quad \nabla \lbrace D^{L} f(x,\mu)(\cdot) \rbrace (y),
\end{align*}
exist and are jointly continuous in the corresponding variables $(x,\mu,y)$, such that $y \in \rm{Supp}(\mu)$.
\item We say that $f \in C^{2,(2,1)}(\mathbb{R}^d \times \mathscr{P}_2(\mathbb{R}^d))$, if $f \in C^{2,(1,1)}(\mathbb{R}^d \times \mathscr{P}_2(\mathbb{R}^d))$ and in addition the second order $L$-derivative $(D^{L})^2f(x,\mu)(y,y')$ exists and is again jointly continuous in the corresponding variables. Also, the joint continuity of all derivatives is required globally, i.e., for all $(x,\mu,y,y')$. In this case, $f(x,\cdot)$ is called fully $\mathcal{C}_2$.
\end{itemize}
For vector-valued or matrix-valued functions, this definition has to be understood componentwise.
\end{defn}
\noindent
We close the discussion on the $L$-derivative by presenting two examples, which are relevant for Section \ref{Section:Sec4}:  \\ \\
\noindent 
\textbf{Example 1:}  \\
Consider $f(\mu) = \int_{\R^{d}} v(x,\mu) \, \mu(\mathrm{d}x)$, for some continuous function $\R^{d} \times \mathscr{P}_2(\R^{d}) \ni (x,\mu) \mapsto v(x,\mu) \in \R$. We assume that for a fixed $\mu$, $v$ is differentiable in $x \in \R^{d}$. The derivative is assumed to be jointly continuous in $(x,\mu)$, and at most of linear growth in $x$, uniformly in $\mu$ in bounded subsets of $\mathscr{P}_2(\R^{d})$. According to  \cite[Example 3 on page 387]{CD} the $L$-derivative of $f$ is $D^{L} f(\mu)(\cdot)=\nabla v(\cdot,\mu) + \int_{\R^{d}} D^{L} v(x',\mu)(\cdot) \, \mu(\mathrm{d}x')$. \\ \\
\noindent 
\textbf{Example 2:} \\
As a second example, we consider $v(x,\mu) = \int_{\R^{d}} g(x,x') \, \mu(\mathrm{d}x')$, where the map $\R^{d} \times \R^{d} \ni (x,x') \mapsto g(x,x') \in \R$ is assumed to be continuously differentiable in $(x,x')$, with partial derivatives being at most of linear growth in $(x,x')$. Then, $D^{L} v(x,\mu)(x') = \nabla g(x,\cdot)(x')$ (see \cite[Example 4 on page 389]{CD}).
\begin{rmk}
We remark that suitable assumptions on, e.g., the function $v(\cdot,\cdot)$ from the last example can be made such that it is in the class of functions introduced in Definition \ref{def:Defin1}. Further, we point out that in the strong convergence analysis for the interacting particle system we only work with empirical measures. In principle, this would allow us to reformulate the differentiability assumptions (for the measure component) made on the coefficients $b$ and $\sigma$ in terms of (classical) differentiability conditions for the empirical projections of $b$ and $\sigma$. However, as we formulate a Milstein scheme for McKean--Vlasov SDEs (without employing the particle method), where we make use of an It\^{o} formula which requires the function to satisfy the conditions presented in the first item of Definition \ref{def:Defin1} (see Section \ref{subsec:lipschitz} for details), we present the differentiability assumptions in such a generality. 
\end{rmk}
  
\noindent
\textbf{Acknowledgement:} Jianhai Bao is supported by National Natural Science Foundation of China (11801406). Wolfgang Stockinger is supported by an Upper Austrian Government fund.

\end{document}